\documentclass{article}

\usepackage[verbose=true,a4paper]{geometry}
\AtBeginDocument{
  \newgeometry{
    textheight=9in,
    textwidth=6.5in,
    top=1in,
    headheight=14pt,
    headsep=25pt,
    footskip=30pt
  }
}

\usepackage{fancyhdr}
\fancyheadoffset{0pt}
\def\keywordname{{\bfseries \emph Keywords}}%
\def\keywords#1{\par\addvspace\medskipamount{\rightskip=0pt plus1cm
\def\and{\ifhmode\unskip\nobreak\fi\ $\cdot$
}\noindent\keywordname\enspace\ignorespaces#1\par}}

\renewenvironment{abstract}
{
  \centerline
  {\large \bfseries \scshape Abstract}
  \begin{quote}
}
{
  \end{quote}
}

\usepackage[numbers]{natbib}
\usepackage[percent]{overpic}

\usepackage{amsmath,amsthm,amssymb}
\usepackage{graphicx,geometry}
\usepackage{xcolor}
\usepackage{float}

\usepackage{fancyhdr}
\usepackage{enumitem}

\newcommand{\Implies}[2]{$\text{\ref{#1}}\implies\text{\ref{#2}}$}

\theoremstyle{definition}
\newtheorem{theo}{Theorem}

\newtheorem{prop}[theo]{Proposition}
\newtheorem{lemma}[theo]{Lemma}

\title{\uppercase{Families of connected self-similar sets generated by Complex Trees}}

\author{
  Bernat Espigule
  \\
  \texttt{bernat@espigule.com}
}
\date{}

\begin{document}

\maketitle

\begin{abstract}
The theory of complex trees is introduced as a new approach to study a broad class of self-similar sets. Systems of equations encoded by complex trees tip-to-tip equivalence relations are used to obtain one-parameter families of connected self-similar sets~$F_\mathcal{A}(z)$. In order to study topological changes of~$F_\mathcal{A}(z)$ in regions~$\mathcal{R}\subseteq\mathbb{C}$ where these families are defined, we introduce a new kind of set~$\mathcal{M}\subseteq\mathcal{R}$ which extends the usual notion of connectivity locus for a parameter space. Moreover we consider another set~$\mathcal{M}_0\subset\mathcal{M}$ related to a special type of connectivity for which we provide a theorem. Among other things, the present theory provides a unified framework to families of self-similar sets traditionally studied as separate with elements~$F_\mathcal{A}(z)$ disconnected for parameters $z\in\mathcal{R}\backslash\mathcal{M}$.
\end{abstract}


\section{Introduction}

In a talk titled "Geometry of Self-similar Sets"~\cite{bandt2014geometry} Christoph Bandt explained how self-similar sets can be treated rigorously, like manifolds in analysis. In particular, he showed that for self-similar sets~$F=f_1(F)\cup f_2(F)$ generated by the pair of mappings
\begin{equation*}
	f_1(z)=c_1z \quad \text{and}\quad f_2(z)=c_2(z-1)+1 \qquad \text{ with }~0<|c_1|,|c_2|<1.
\end{equation*}
one can establish conditions for the existence of an intersection point~$x\in f_1(F)\cap f_2(F)$ accessible from the coding addresses $111\overline{2}$~and~$211\overline{2}$,~i.e. $f_1 f_1 f_1 f_2 f_2\ldots(z)=x=f_2 f_1 f_1 f_2 f_2\ldots(z)$. He showed that for $c_2=1+c_1^2/(c_1+1)$ the equivalence relation~$111\overline{2}\sim211\overline{2}$ is being satisfied. As he noted, parameters $c_1$~and~$c_2$ need to be picked carefully in order to avoid further intersection points~$x'\in f_1(F)\cap f_2(F)$. Motivated by Bandt's observation, we will establish conditions for the existence of such parameters in families of self-similar sets~$F$ generated by $n$~mappings
\begin{equation}\label{ifs}
	f_1(z)=1+c_1z \text{ , }\quad f_2(z)=1+c_2z \text{ ,}\quad\text{\dots , }\quad f_n(z)=1+c_nz \qquad \text{ with }~0<|c_1|,|c_2|,\dots,|c_n|<1.
\end{equation}
The restriction to this kind of iterated function system (IFS) is due to a geometric interpretation that arises naturally. A self-similar set~$F$ generated by this system is completely encoded by the limiting points of what we call complex tree, a special type of fractal tree in the plane. The idea behind the notion of complex tree appeared to the author as a combinatorial generalization of the geometric series~$1+z+z^2+z^3+...$. When~$|z|<1$ the series always converges and their partial sums form a sequence of points that spiral in towards $1/(1-z)$. If instead of multiplying by~$z$ each successive summand, $1, z, z^2,\dots$, we consider all the possible choices of a collection of $n$ complex numbers $\{c_1,c_2,\dots,c_n\}=\mathcal{A}$ 
we will end up with a \textbf{complex tree} $T_\mathcal{A}$. By definition we set the~$n$ complex-valued letters of the alphabet~$\mathcal{A}=\{c_1,c_2,\dots,c_n\}$ to be different to each other and with positive absolute value smaller than 1. Our main tool is the parameterization of~$T_\mathcal{A}$ in terms of complex points~$\phi(w)$ generated by projecting words~$w:=w_1w_2\dots w_m$, with letters $w_1, w_2,\dots, w_m \in \mathcal{A}$, through the~\textbf{geometric~map}~$\phi$ defined as:
\begin{equation}
	\phi(w):=1+w_1+w_1w_2+w_1w_2w_3+\dots+w_1w_2w_3\dots w_m.
\end{equation}
When a word~$w=w_1w_2w_3\dots w_m$ has a finite number of letters, $\phi(w)$ is called a \textbf{node} of the complex tree~$T_\mathcal{A}$. But when a word~$w=w_1 w_2\dots w_k\dots$ has infinitely many letters taken from the~$n$-ary alphabet~$\mathcal{A}=\{c_1,c_2,\dots,c_n\}$ then $\phi(w)$ is called a \textbf{tip point} of the complex tree~$T_\mathcal{A}$ and we express it as an infinite sum
 \begin{equation}
 	\phi(w)=1+w_1+w_1w_2+w_1w_2w_3+\dots+w_1w_2w_3\dots w_k+\dots=\sum_{k=0}^\infty w_{|k}
 \end{equation}
where each summand $w_{|k}=w_1w_2\dots w_k$ is assumed to be the complex multiplication of the individual letters of the word~$w$ pruned up to its $k$th letter~$w_k$. For $k=0$, we have that $w_{|0}=w_0=e_0$ where~$e_0$ is the~\textbf{empty string} with assigned value equal to~1. The node $\phi(e_0)=1$ is the~\textbf{root} of our $n$-ary complex trees where their first $n$~nodes, $\phi(c_1)=1+c_1$, $\phi(c_2)=1+c_2$,~\dots,~$\phi(c_n)=1+c_n$, sprout from. 
\begin{figure}[H]
\begin{overpic}[width=5.7in]{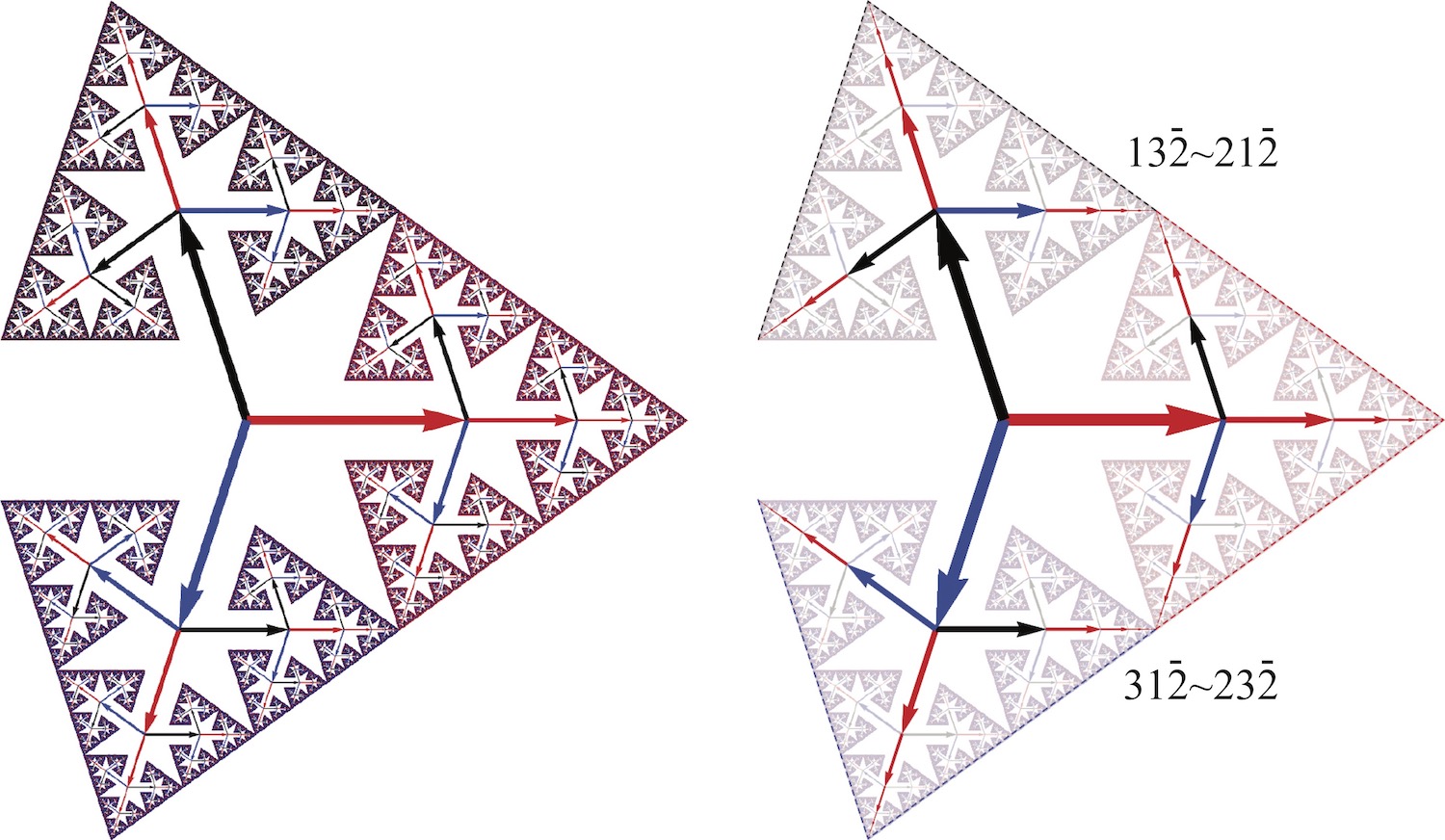}
\put(58,29){$\phi(e_0)=1$}
\put(58.7,44.1){$\phi(c_1)$}
\put(58.7,13.1){$\phi(c_3)$}
\put(84.7,30.2){$\phi(c_2)$}
\end{overpic}
\centering
\caption{Ternary tree $T\{c_1,c_2,c_3\}$ with complex-valued letters $c_1=i^{6/5}/2$, $c_2=1/2$, and $c_3=-i^{4/5}/2$.
}
\label{upternary}
\end{figure}
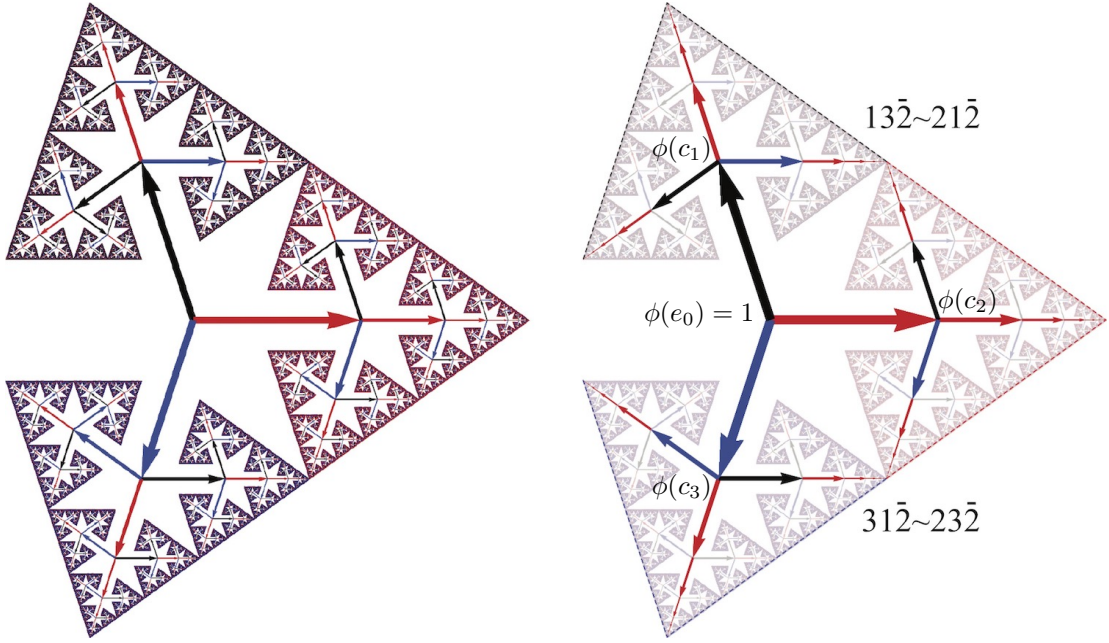
\noindent By imposing a color code, $\{1\rightarrow c_1, 2\rightarrow c_2,\dots,n\rightarrow c_n\}$, words~$u=u_1u_2\dots u_m$ can be retrieved by reading the color sequence from the root~$\phi(e_0)$ to the desired node~$\phi(u)$. From now on, letters~$w_k\in\{c_1,c_2,\dots,c_n\}$ found in words will be replaced by numeric symbols~$\{1,2,\dots,n\}$ to facilitate reading, for example $c_1c_3c_2c_2=1322$. Infinite sequence of letters $w_1,w_2,\dots,w_k,\dots\in\mathcal{A}$ that are eventually periodic have their associated tip point~$\phi(w)$ reduced to algebraic expressions in terms of $c_1,c_2,\dots,c_n\in\mathcal{A}$. For example, pairs of tip points meeting at the same point with their branch-path highlighted in figure~\ref{upternary} have the following closed forms
\begin{equation*}
\phi(13222\dots)=\phi(13\overline{2})=1+c_1+c_1c_3+c_1c_3c_2+\dots=1+c_1+c_1c_3\sum_{k=0}^\infty c_2^k=1+c_1+\frac{c_1c_3}{1-c_2},
\end{equation*}
\begin{equation*}
\phi(21\overline{2})=1+c_2+\frac{c_2c_1}{1-c_2},\quad \phi(23\overline{2})=1+c_2+\frac{c_2c_3}{1-c_2}, \quad \phi(31\overline{2})=1+c_3+\frac{c_3c_1}{1-c_2}.
\end{equation*}
Let $\mathcal{A}^m:=\{v=v_1v_2\dots v_m : v_1,v_2,\dots, v_m\in\mathcal{A}\}$ denote the set of finite words of length~$m$ over the alphabet~$\mathcal{A}$. Let $\mathcal{A}^*:=\cup_{m\geq 0}\mathcal{A}^m$ be the set of finite words of any length. Let $\mathcal{A}^\infty:=\{w=w_1w_2\dots : w_1,w_2,\ldots\in\mathcal{A}\}$ denote the set of all infinite words. And let $\Omega:=\mathcal{A}^*\cup\mathcal{A}^\infty$ be the set of all finite and infinite words. The geometric map $\phi$ has the following property. Let $v\in \mathcal{A}^m$ and $w\in \Omega$ then:
\begin{equation}\label{vw}
	\begin{aligned}
\phi(vw)&=1+v_1+v_1v_2+\dots+v_1v_2\dots v_m+(v_1v_2\dots v_m)w_1+\dots  \\
	&=\phi(v)+(v_1v_2\dots v_m )w_1+(v_1v_2\dots v_m )w_1w_2+\dots \\
	&=\phi(v)+(v_1v_2\dots v_m)\cdot(w_1+w_1w_2+\dots) \\
	&=\phi(v)+v\cdot(w_1+w_1w_2+\dots) \\
	&=\phi(v)+v\cdot(\phi(w)-1).
	\end{aligned}
\end{equation}
For a given $n$-ary complex tree~$T_\mathcal{A}$ we define its~\textbf{tipset}~$F_\mathcal{A}$ as the set of all tip points~$\phi(w)$ over the alphabet~$\mathcal{A}$, i.e.
\begin{equation}
	F_\mathcal{A}:=\{\phi(w):w\in \mathcal{A}^\infty\}=\bigcup_{w\in \mathcal{A}^\infty}\phi(w).
\end{equation}
From~(\ref{vw}) we have that a~\textbf{piece}~$F_{v\mathcal{A}}$ of level~$m$, defined as the set of all tip points~$\phi(vw)$ with $v$ fixed for a given finite word~$v=v_1v_2\dots v_m\in\mathcal{A}^m$, can be written in terms of the tipset~$F_\mathcal{A}$ as 
\begin{equation}
	F_{v\mathcal{A}}=\bigcup_{w\in \mathcal{A}^\infty}\phi(vw)=\phi(v)+v\cdot(\bigcup_{w\in \mathcal{A}^\infty}\phi(w)-1)=\phi(v)+v\cdot(F_\mathcal{A}-1)=f_v(F_\mathcal{A}),\label{piece}
\end{equation}
where~$f_v$ is a similarity map defined as~$f_v(z):=\phi(v)+v\cdot(z-1)$. Moreover, we have that the tipset~$F_\mathcal{A}$ of an $n$-ary tree admits a decomposition in~$n^m$ pieces $F_{v\mathcal{A}}$ of level~$m$, i.e. 
\begin{equation}
	F_\mathcal{A}=\bigcup_{w\in\mathcal{A}^\infty}\phi(w)=\bigcup_{v\in\mathcal{A}^m} \bigcup_{w\in \mathcal{A}^\infty}\phi(vw)=\bigcup_{v\in\mathcal{A}^m}F_{v\mathcal{A}}.
\end{equation}
In particular, for $m=1$ we have that $\{f_v\}_{v\in\mathcal{A}}=\{f_1(z)=1+c_1z,f_2(z)=1+c_2z,\dots,f_n(z)=1+c_nz\}$ satisfies the~\textbf{self-similarity equation} introduced by Hutchinson~\cite{hutchinson1981fractals},
\begin{equation}
	F_\mathcal{A}=\bigcup_{v\in\mathcal{A}^1}F_{v\mathcal{A}}=F_{1\mathcal{A}}\cup F_{2\mathcal{A}}\cup\dots\cup F_{n\mathcal{A}}=f_1(F_\mathcal{A})\cup f_2(F_\mathcal{A})\cup\dots\cup f_n(F_\mathcal{A}).
\end{equation}
So a tipset~$F_\mathcal{A}$ is the invariant set generated by the IFS introduced in~(\ref{ifs}) which consists of~$n$~mappings~$\{f_1,f_2,\dots,f_n\}$ where the parameters $c_1,c_2,\dots,c_n$ are the complex-valued letters of the alphabet~$\mathcal{A}$. See example below. 
\begin{figure}[H]
\includegraphics[width=5.7in]{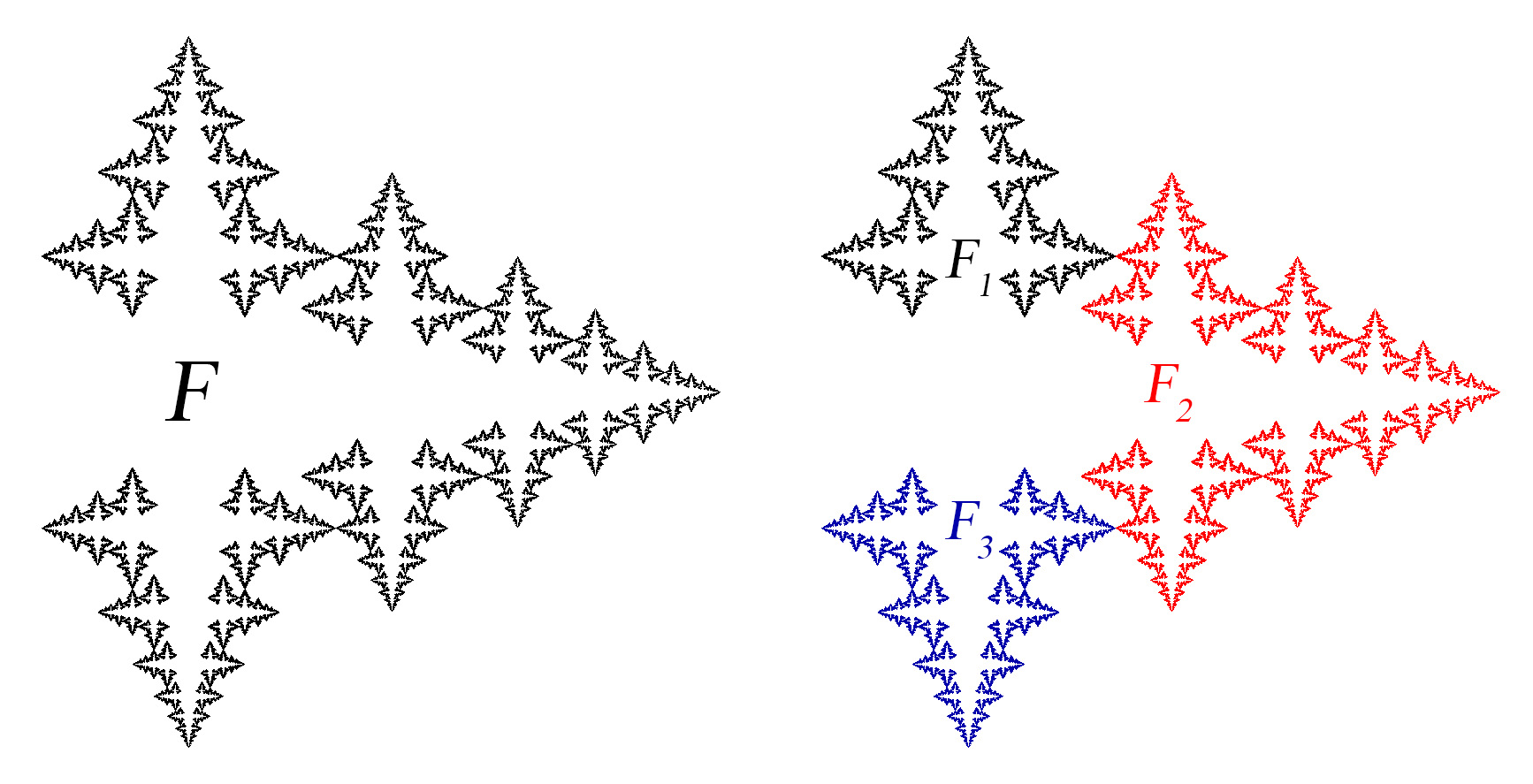}
\centering
\caption{Tipset~$F_\mathcal{A}$ with its first-level pieces~$F_{1\mathcal{A}}$, $F_{2\mathcal{A}}$,~and~$F_{3\mathcal{A}}$ on the right. Its ternary alphabet $\mathcal{A}=\{c_1,c_2,c_3\}$ has complex-valued letters $c_1=i/(\tau\sqrt{\sqrt{5}})$, $c_2=1/\tau$, and $c_3=-i/(\tau\sqrt{\sqrt{5}})$ where~$\tau$ is the golden ratio.
}
\label{sstipset}
\end{figure}
\noindent The self-similar nature of the tipset~$F_\mathcal{A}$ implies that pieces~$F_{v\mathcal{A}}$ of any level~$m$ are exact smaller copies of~$F_\mathcal{A}$. Hence, knowing how the first-level pieces $F_{1\mathcal{A}}, F_{2\mathcal{A}},\dots, F_{n\mathcal{A}}$ intersect each other is all we need in order to capture the topological structure of~$F_\mathcal{A}$. This information is given by what we call the \textbf{topological set}~$Q_\mathcal{A}$ defined as the set of all tip-to-tip equivalence relations $a\sim b$ where $a,b\in\mathcal{A}^\infty$, $a_1\neq b_1$, and~$\phi(a)=\phi(b)$, i.e.
\begin{equation}
	Q_\mathcal{A}:=\{a\sim b :\phi(a)=\phi(b)\text{ for } a,b\in \mathcal{A}^\infty\text{ with }a_1\neq b_1 \}.
\end{equation}
For example, the topological sets~$Q_\mathcal{A}$ for trees depicted in figures~\ref{upternary}~and~\ref{unstable-sierpinski} are $Q_\mathcal{A}=\{13\overline{2}\sim 21\overline{2},31\overline{2}\sim 23\overline{2}\}$ and~$Q_\mathcal{A}=\{11\overline{2}\sim33\overline{2},13\overline{2}\sim 21\overline{2},31\overline{2}\sim 23\overline{2}\}$ respectively. 
Notice that if we have that a tip-to-tip equivalence relation $a\sim b\in Q_\mathcal{A}$ takes place then eq.~\ref{vw} implies that $\phi(ua)=\phi(u)+u\cdot(\phi(a)-1)=\phi(u)+u\cdot(\phi(b)-1)=\phi(ub)$ for all $u\in \mathcal{A}^*$.

\begin{figure}[H]
\includegraphics[width=5.5in]{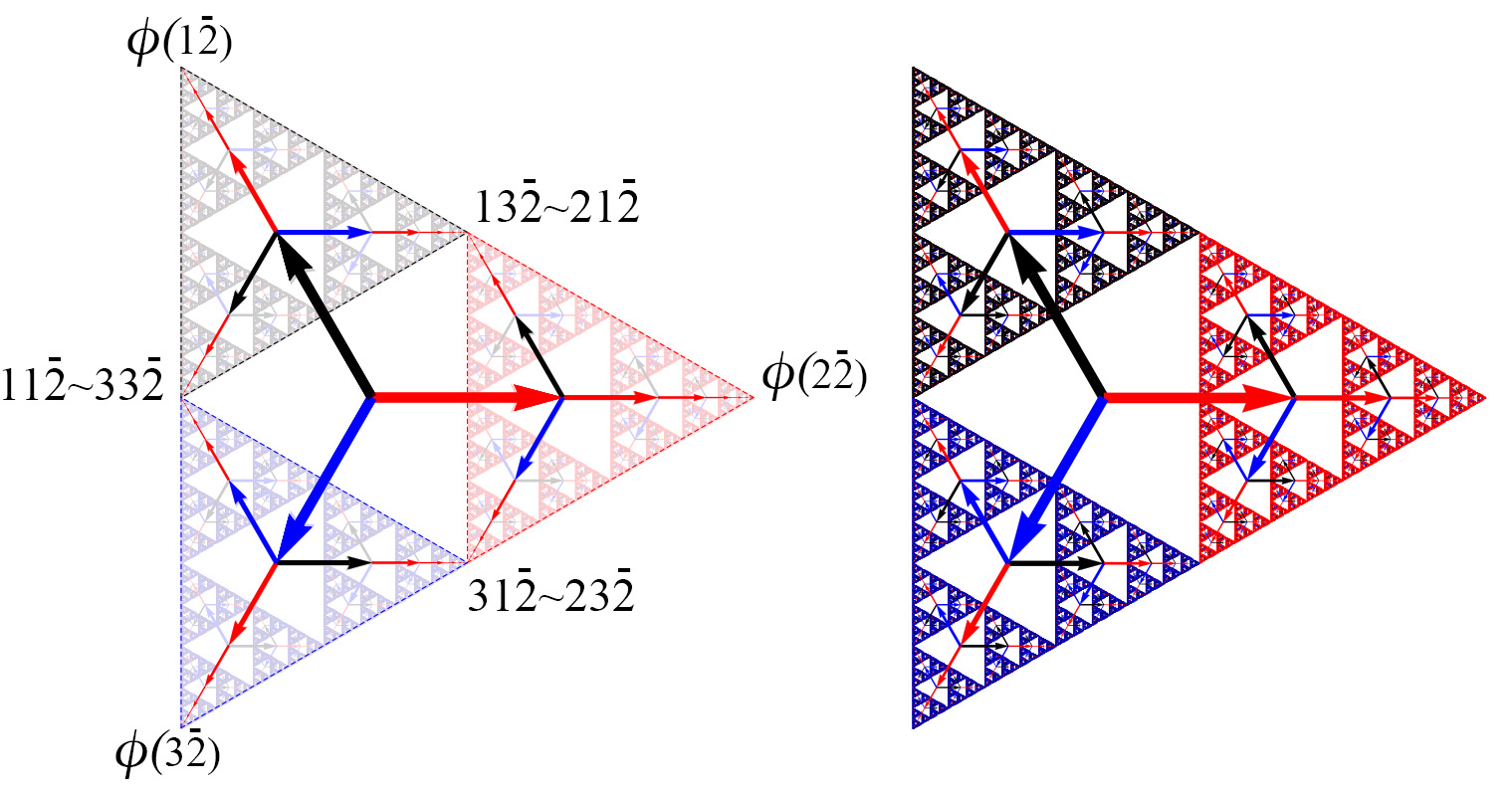}
\centering
\caption{Ternary complex tree $T\{(-1+i \sqrt{3})/4,1/2,(-1-i \sqrt{3})/4\}$ associated to the Sierpinski triangle.}
\label{unstable-sierpinski}
\end{figure}
\begin{figure}[H]
\includegraphics[width=4.4in]{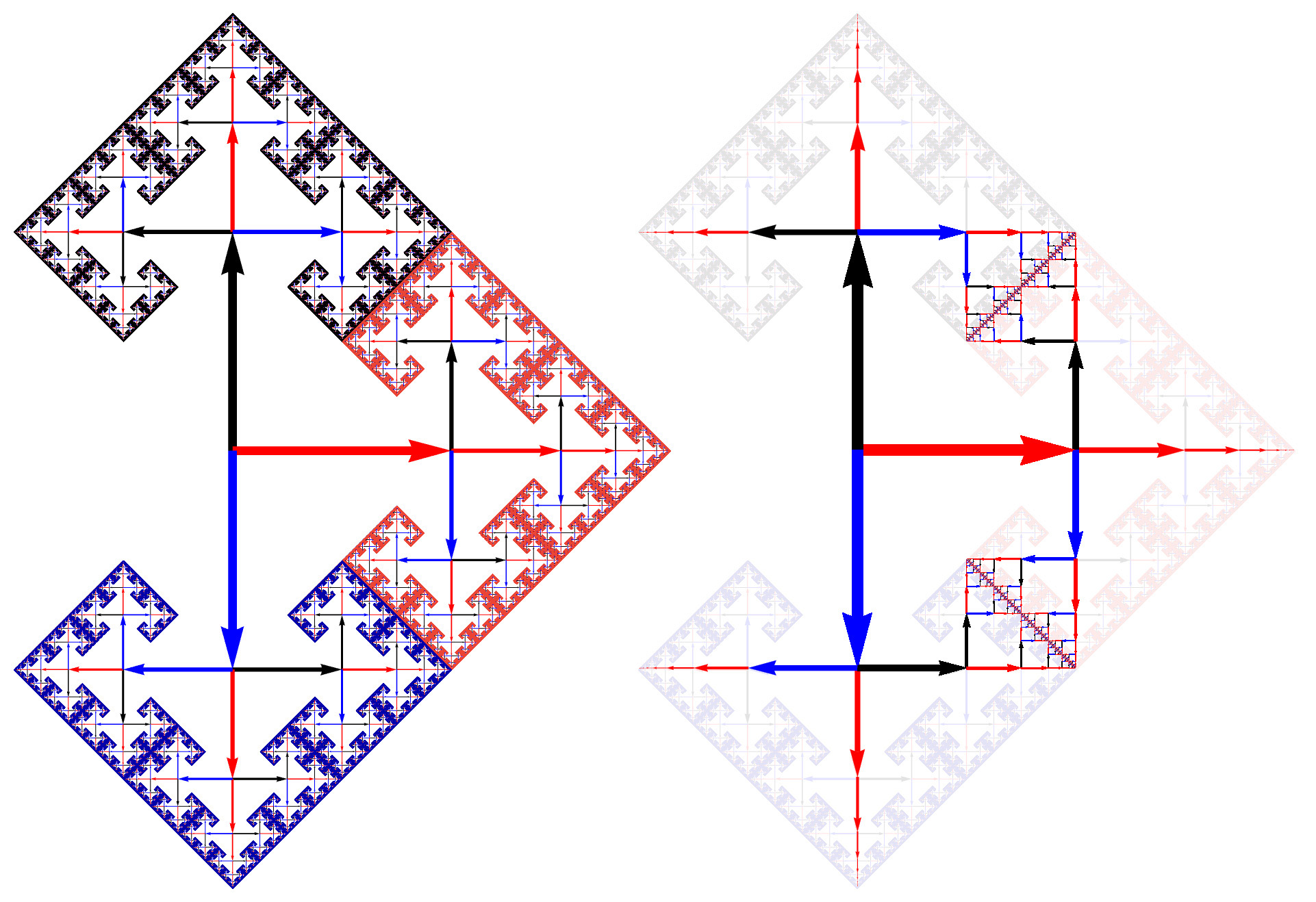}
\centering
\caption{Mirror-symmetric ternary complex tree $T_\mathcal{A}=T\{i/2,1/2,-i/2\}=T\{-i/2,1/2,i/2\}$ with first-level pieces intersecting in a pair of straight segments that go from $\phi(211\overline{2})$ to $\phi(21\overline{2})$ for $F_{1\mathcal{A}}\cap F_{2\mathcal{A}}$, and $\phi(233\overline{2})$ to $\phi(23\overline{2})$ for $F_{2\mathcal{A}}\cap F_{3\mathcal{A}}$. 
}
\label{i2}
\end{figure}

\section{Families of Connected Self-similar Sets}

Given a topological set $Q_\mathcal{A}$ of a tipset-connected $n$-ary complex tree $T_\mathcal{A}$ consider all tip-to-tip equivalence relations $a\sim b\in Q_\mathcal{A}$ as a system of equations $\phi(a)=\phi(b)$ with letters of the alphabet $\mathcal{A}=\{c_1,c_2,\dots,c_n\}$ set as $n$~unknown complex variables. Let's assume $\mathbf{card}(Q_\mathcal{A})=n-1$. Then we have a system of $n-1$~equations with $n$~unknowns that admits a parametric solution in one or several complex variables depending on the tipset-connectivity. For example, the ternary complex tree in figure~\ref{upternary} has a topological set, $Q_\mathcal{A}=\{13\overline{2}\sim 21\overline{2},23\overline{2}\sim 31\overline{2}\}$, that translates to a system of two equations with three unknowns that admits a reduction with~$c_1$ as the only unknown variable: 
\begin{align}\label{z 1/2 1/4z}
	\textbf{sol}(Q_\mathcal{A}):=
	\begin{cases}
	\phi(13\overline{2})=\phi(21\overline{2})\\
	\phi(23\overline{2})=\phi(31\overline{2})
	\end{cases}
	=\begin{cases}
	1+c_1+\frac{c_1c_3}{1-c_2}=1+c_2+\frac{c_2c_1}{1-c_2}\\
	1+c_2+\frac{c_2c_3}{1-c_2}=1+c_3+\frac{c_3c_1}{1-c_2}
	\end{cases}
	\Longrightarrow 	
	\begin{cases}
	c_2=1/2\\
	c_3=1/4c_1
	\end{cases}.
\end{align}
Assigning the variable~$z$ to $c_1$, the parametric family~$T_\mathcal{A}(z)=T\{c_1(z),c_2(z),c_3(z)\}$ obtained from $\textbf{sol}(Q_\mathcal{A})$ is $T_\mathcal{A}(z)=T\{z,1/2,1/4z\}$. The topological set of the tree depicted below is $Q_\mathcal{A}=\{12\overline{31}\sim 22\overline{13},32\overline{13}\sim 22\overline{31}\}$, and $\textbf{sol}(Q_\mathcal{A})$ provides a slightly different family, $T_\mathcal{A}(z)=T\{z,-1/2,1/4z\}$.
\begin{equation}\label{z -1/2 1/4z}
	\textbf{sol}(Q_\mathcal{A}):=
	\begin{cases}
	\phi(12\overline{31})=\phi(22\overline{13})\\
	\phi(32\overline{13})=\phi(22\overline{31})
	\end{cases}
=\begin{cases}
	c_1+c_1c_2(\frac{1+c_3}{1-c_1c_3})=c_2+c_2^2(\frac{1+c_1}{1-c_3c_1})\\
	c_3+c_3c_2(\frac{1+c_1}{1-c_1c_3})=c_2+c_2^2(\frac{1+c_3}{1-c_3c_1})
	\end{cases}
	\Longrightarrow 
	\begin{cases}
	c_2=-1/2\\
	c_3=1/4c_1
	\end{cases}.
\end{equation}
\begin{figure}[H]
\includegraphics[width=6in]{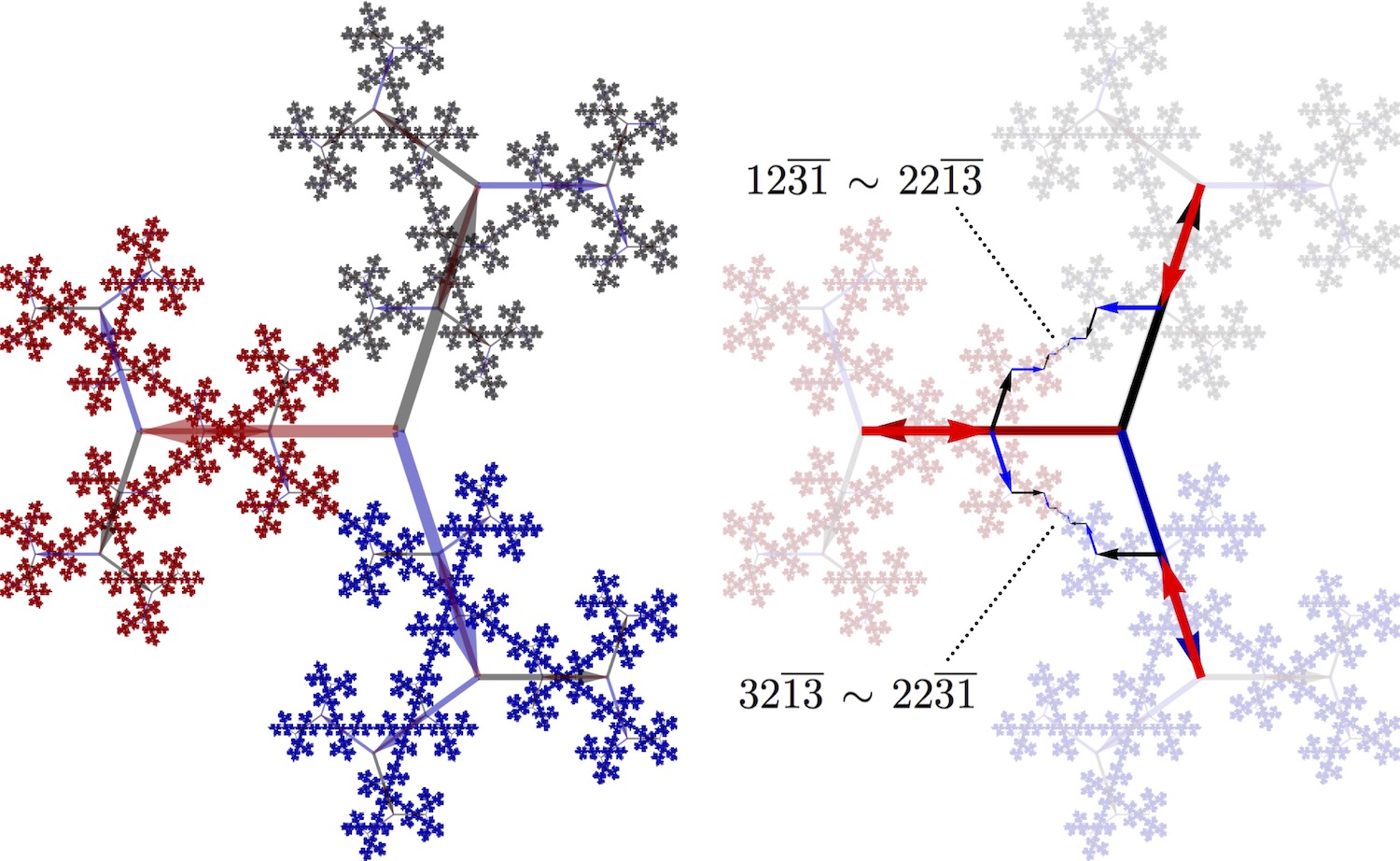}
\centering
\caption{Ternary tree $T\{c_1,c_2,c_3\}$ with complex-valued letters $c_1=i^{4/5}/2$, $c_2=-1/2$, and $c_3=-i^{6/5}/2$.}
\label{downternary}
\end{figure}
\noindent The following families of tipset-connected binary trees $T\{z,1+z^2\}$, $T\{z,z+1/(1+z)\}$, $T\{z,1+z+z^2\}$
 were obtained from topological sets of trees with a single tip-to-tip equivalence relation, see figure~\ref{regionsbinary}. Their $\textbf{sol}(Q_\mathcal{A})$ get reduced to the following algebraic expressions, recall that by definition $0<|c_1|,|c_2|<1$ and $c_1\neq c_2$,
\begin{align}\label{z 1+z^2}
	\textbf{sol}(\{1111\overline{2}\sim 211\overline{2}\}):&=
	\begin{cases}\phi(1111\overline{2})=\phi(211\overline{2})\end{cases}\notag\\
	&=\begin{cases}1+c_1+c_1^2+c_1^3+\frac{c_1^4}{1-c_2}=1+c_2+c_2 c_1+\frac{c_2 c_1^2}{1-c_2}\end{cases}\\
	&=\begin{cases}(c_1+1)(c_1-c_2)(c_1^2-c_2+1)/(c_2-1)=0\end{cases} \notag\\
	&\text{ i.e.}\quad c_1^2-c_2+1=0  \quad\Longrightarrow\quad c_2=1+c_1^2.\notag
 \end{align}
 \begin{align}
	\textbf{sol}(\{111\overline{2}\sim 211\overline{2}\}):&=
	\begin{cases}\phi(111\overline{2})=\phi(211\overline{2})\end{cases}\quad\Longrightarrow\quad c_2=c_1+1/(1+c_1).\label{z z+1/(1+z)}\\
	\textbf{sol}(\{111\overline{2}\sim 21\overline{2}\}):&=
	\begin{cases}\phi(111\overline{2})=\phi(21\overline{2})\end{cases}\quad\Longrightarrow\quad c_2=1+c_1+c_1^2.\label{z 1+z+z^2}
 \end{align}
 \begin{figure}[H]
 \centering
 \begin{overpic}[width=1.01\textwidth,tics=10]{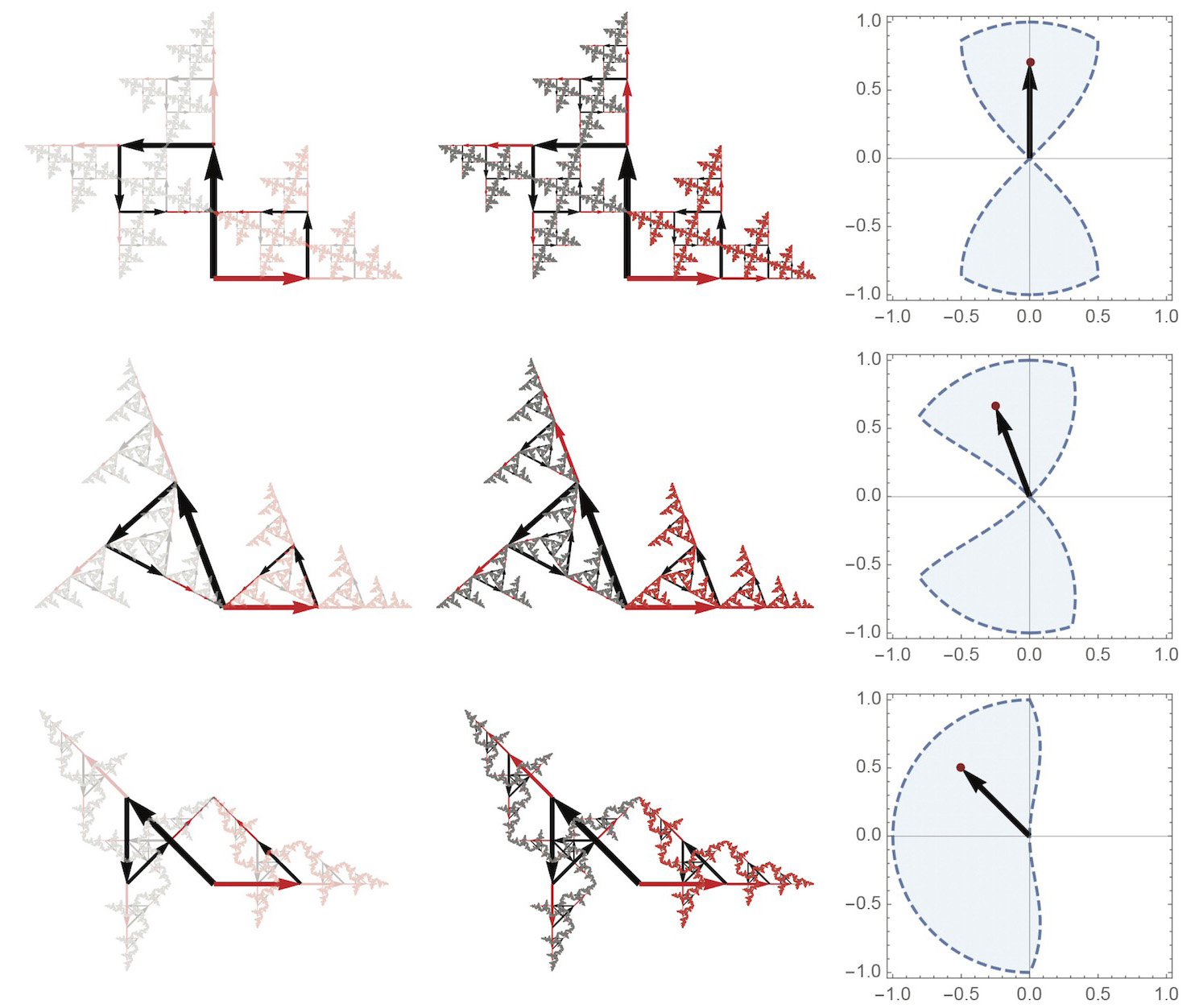}
 \put(19,72) {$1111\overline{2}\sim211\overline{2}$}
 \put(58,77) {$T\{z,1+z^2\}$}
 \put(60,73) {$z=i/\sqrt{2}$}
 \put(16,45) {$111\overline{2}\sim211\overline{2}$}
 \put(50,50) {$T\{z,z+1/(1+z)\}$}
 \put(52,46) {$z=(-1+i \sqrt{7})/4$}
 \put(12,20) {$111\overline{2}\sim21\overline{2}$}
 \put(50,25) {$T\{z,1+z+z^2\}$}
 \put(52,21) {$z=(-1+i)/2$}
\end{overpic}
  \caption{Left column, branch-paths of tip-to-tip equivalence relations $a\sim b$ where $F_{1\mathcal{A}}\cap F_{2\mathcal{A}}=\{\phi(a)=\phi(b)\}$. Central column, binary complex trees $T\{i/\sqrt{2},1/2\}$, $T\{(-1+i \sqrt{7})/4,1/2\}$, and $T\{(-1+i)/2,1/2\}$. Right column, regions $\mathcal{R}:=\{z\in\mathbb{C}: 0<|z|,|c_2(z)|<1\}$ where one-parameter families of complex trees~$T_\mathcal{A}(z)=T\{z,c_2(z)\}$ are defined. From top to bottom the algebraic expression for $c_2(z)$ was obtained in~(\ref{z 1+z^2}), (\ref{z z+1/(1+z)}), and~(\ref{z 1+z+z^2}) respectively. 
    }
  \label{regionsbinary}
\end{figure}
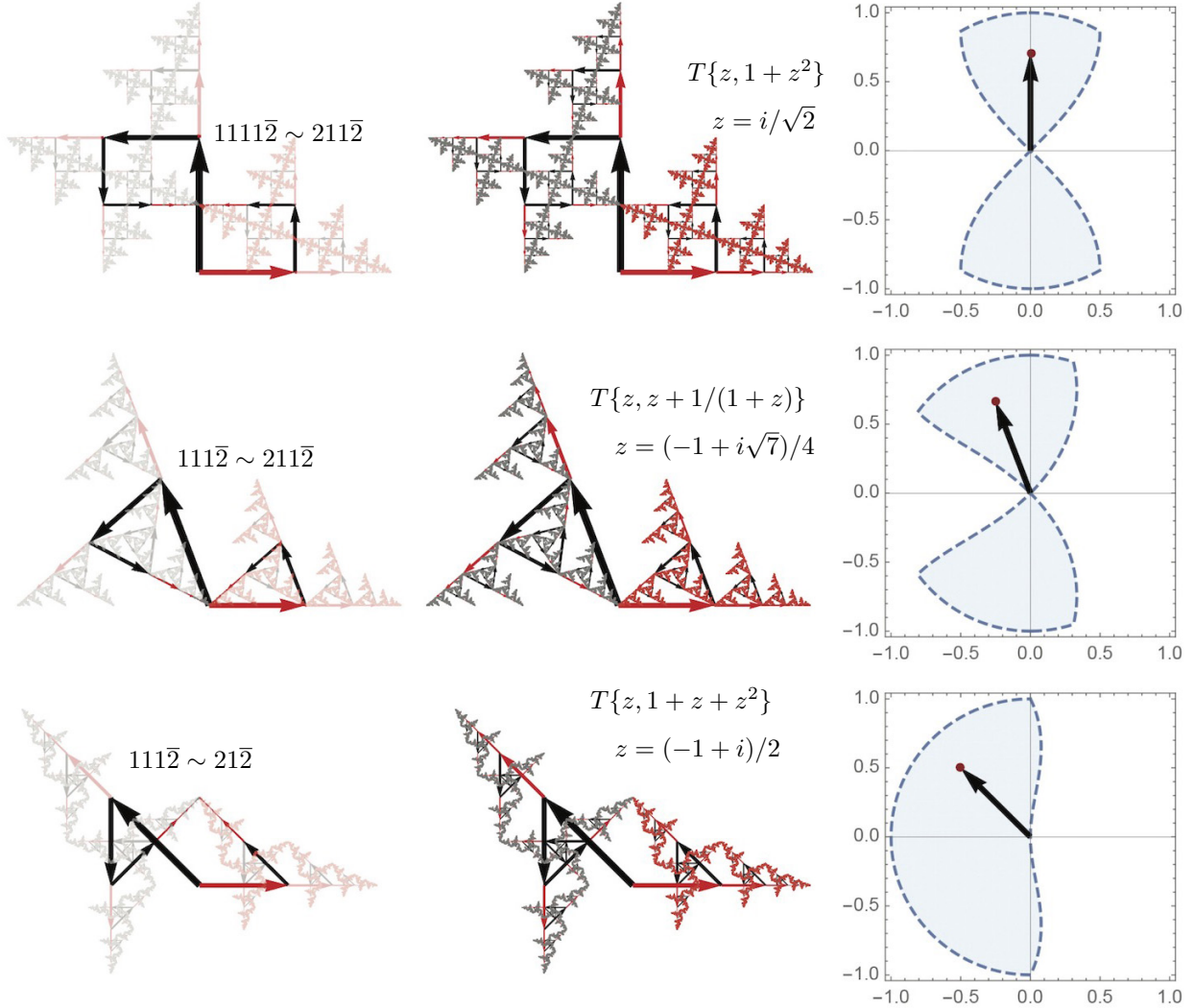
\noindent Self-similar sets associated to these families of binary complex trees $T_\mathcal{A}(z)=T\{c_1(z),c_2(z)\}$ are always connected because $Q_\mathcal{A}(z)\neq\emptyset$ contains at least a tip-to-tip equivalence relation $a\sim b$, i.e.~$\phi(a),\phi(b)\in F_{1\mathcal{A}(z)}\cap F_{2\mathcal{A}(z)}\neq\emptyset$ and by self-similarity the whole tipset $F_\mathcal{A}(z)$ is connected.
\\

It is important to notice that for $\mathbf{card}(Q_\mathcal{A})\geq n$ the solution might still lead to a parametric alphabet in one or several complex variables. For example the ternary complex tree in figure~\ref{stablegasket} has five tip-to-tip equivalence relations, i.e.~$\mathbf{card}(Q_\mathcal{A})=5$, with topological set $Q_\mathcal{A}=\{111\overline{2}\sim331\overline{2}, 131\overline{2}\sim211\overline{2}, 13\overline{2}\sim23\overline{2}, 13\overline{2}\sim33\overline{2}, 33\overline{2}\sim23\overline{2}\}$, but $\mathbf{sol}(Q_\mathcal{A})$ still gives us a parametric family of ternary complex trees~$T_\mathcal{A}(z)$ in one complex variable.
 \begin{figure}[H]
 \centering
 \begin{overpic}[width=5.7in,tics=10]{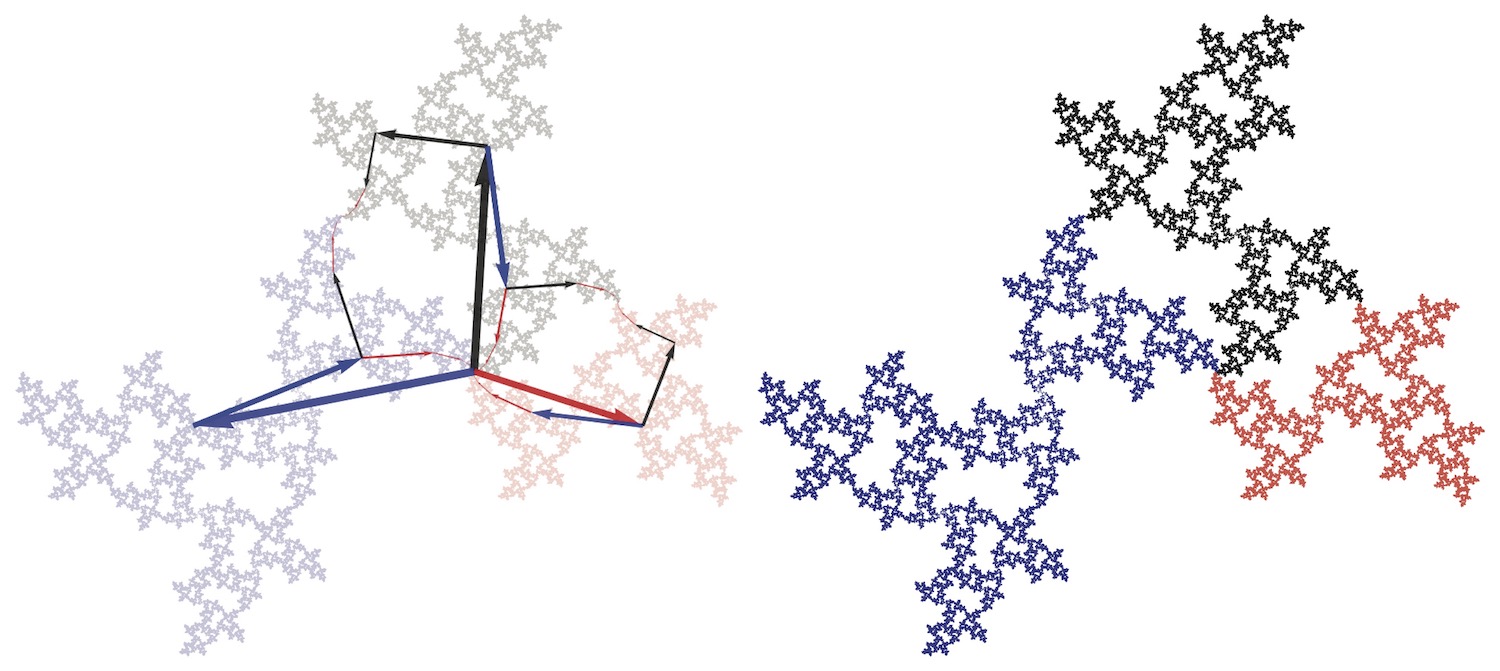}
 \put(8,31) {$111\overline{2}\sim331\overline{2}$}
 \put(42,26) {$131\overline{2}\sim211\overline{2}$}
 \put(23,15) {$13\overline{2}\sim23\overline{2}$}
 \put(23,11) {$13\overline{2}\sim33\overline{2}$}
 \put(23,7) {$33\overline{2}\sim23\overline{2}$}
 \put(77,8){$F_\mathcal{A}$}
\end{overpic}
 \begin{overpic}[width=5.7in,tics=10]{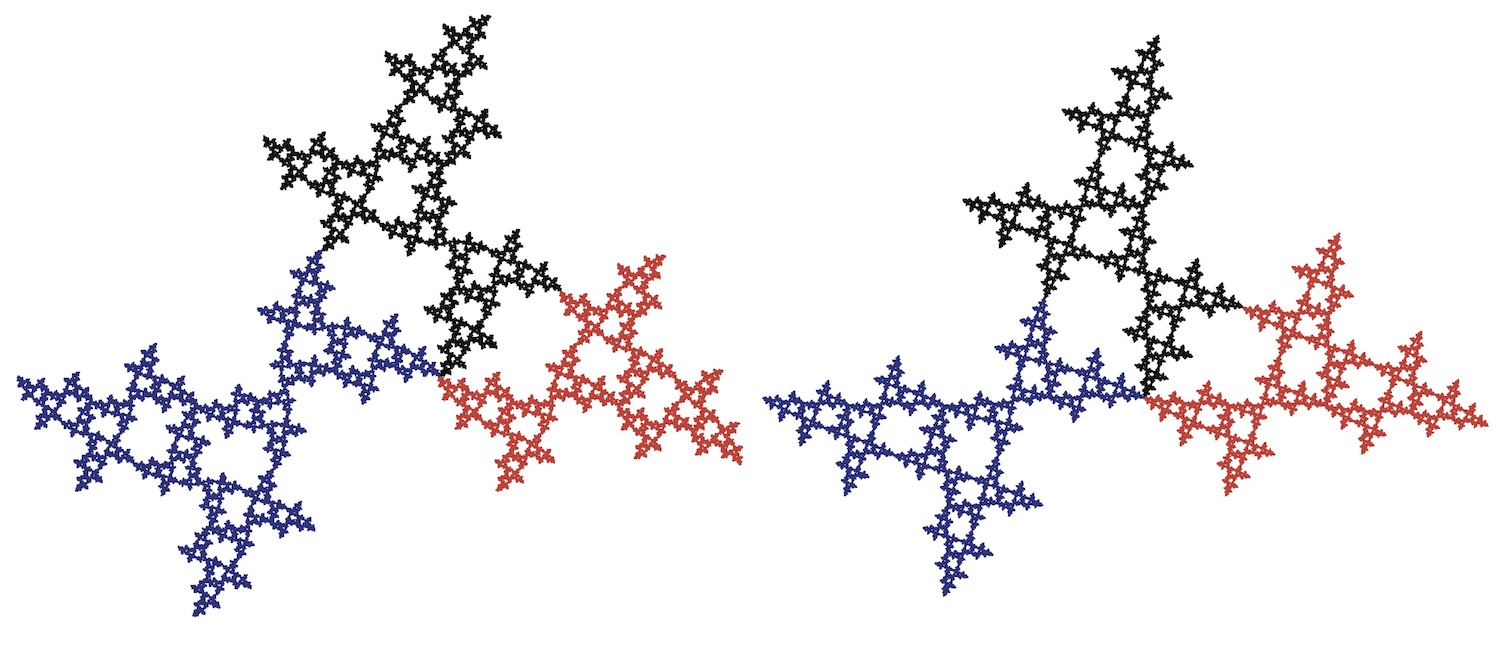}
  \put(25,8){$F_\mathcal{A'}$}
  \put(73,8){$F_\mathcal{A''}$}
  \end{overpic}
  \caption{Branch-paths of tip-to-tip equivalence relations $a\sim b\in Q_\mathcal{A}$ for a ternary complex tree~$T_\mathcal{A}$ with alphabet $\mathcal{A}=\{0.03 +i0.5,0.376 -i0.121,-0.624-i0.121\}$. Below, pair of tipsets~$F_\mathcal{A'}$~and~$F_\mathcal{A''}$ topologically homeomorphic to~$F_\mathcal{A}$, i.e. their share the same topological set $\{111\overline{2}\sim331\overline{2}, 131\overline{2}\sim211\overline{2}, 13\overline{2}\sim23\overline{2}, 13\overline{2}\sim33\overline{2}, 33\overline{2}\sim23\overline{2}\}$.
    }
  \label{stablegasket}
\end{figure}
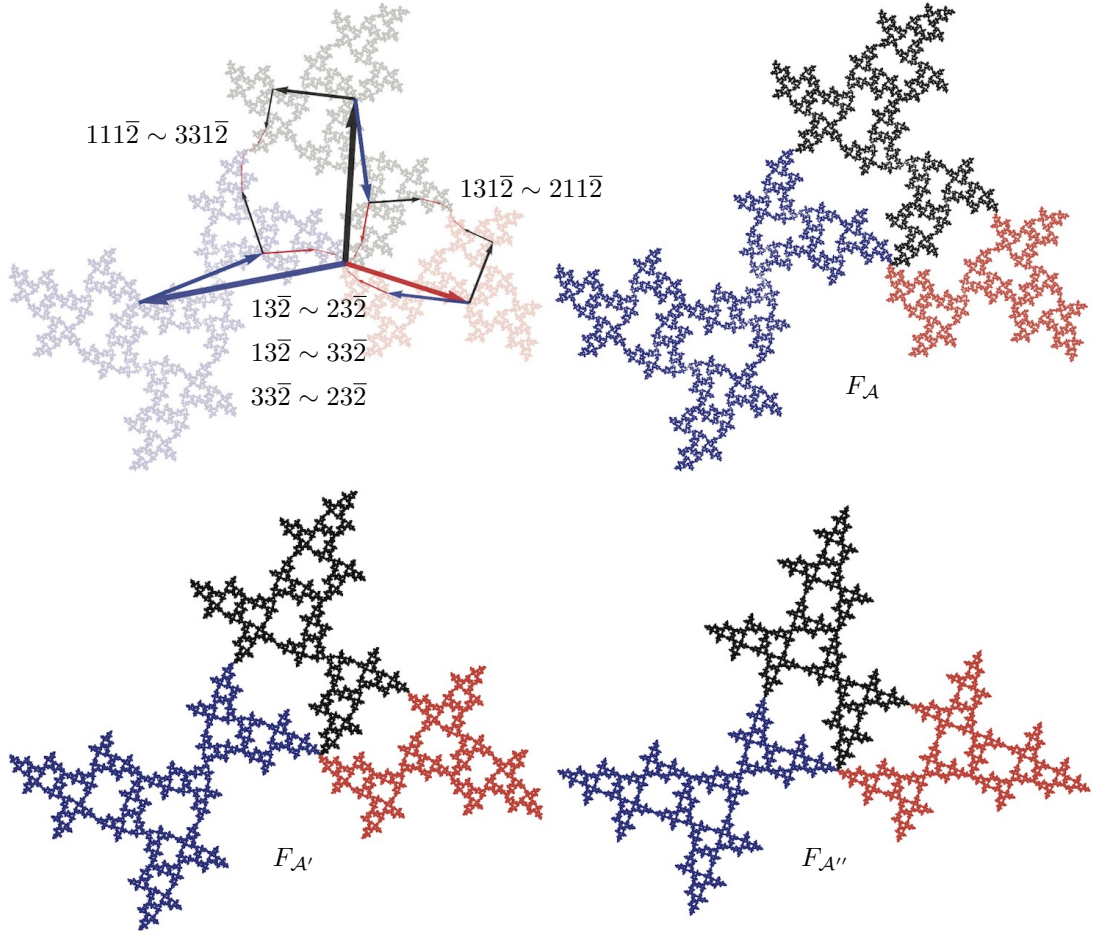

\subsection{The Unstable Set~$\mathcal{M}$}\label{M}
The tipset-connected tree~$T_\mathcal{A}$ in figure~1 with $\mathbf{sol}(Q_\mathcal{A})$ reduced to a parametric alphabet~$\mathcal{A}(z)=\{z,1/2,1/4z\}$ in one complex variable implies that there exists an epsilon-neighborhood around $z=i^{6/5}/2$ where trees~$T_\mathcal{A}(i^{6/5}/2+\epsilon)$ share the same topological set~$\{13\overline{2}\sim 21\overline{2},23\overline{2}\sim 31\overline{2}\}=Q_\mathcal{A}=Q_\mathcal{A}(i^{6/5}/2+\epsilon)$. We call such a tree \textbf{structurally stable}. On the other hand, if for a given tree~$T_\mathcal{A}$ such an epsilon-neighborhood does not exist we call $T_\mathcal{A}$ \textbf{structurally unstable}. As an example consider the ternary complex tree in figure~\ref{unstable-sierpinski} with topological set $Q_\mathcal{A}=\{11\overline{2}\sim 33\overline{2},13\overline{2}\sim 21\overline{2},23\overline{2}\sim 31\overline{2}\}$. Its tipset is the Sierpinski triangle and $\mathbf{sol}(Q_\mathcal{A})$ gets reduced to a numerical solution, 
$\mathcal{A}=\{(-1+i \sqrt{3})/4,1/2,(-1-i \sqrt{3})/4\}$, so any perturbation of its alphabet $\mathcal{A}$ will destroy the original topological set,~$Q_\mathcal{A}$. 
The dichotomy between structurally stable and structurally unstable complex trees opens the possibility to consider the analytic region~$\mathcal{R}:=\{z:0<|c_1(z)|,|c_2(z)|,\dots,|c_n(z)|<1 \}$ for which a given family $T_\mathcal{A}(z)=T\{c_1(z),c_2(z),\dots,c_n(z)\}$ is defined, as a disjoint union of two different regions, the \textbf{stable set}~$\mathcal{K}$ and the \textbf{unstable set}~$\mathcal{M}$ defined~as  
\begin{align}
  	\mathcal{K}:&=\{z\in\mathcal{R}: Q_\mathcal{A}=Q_\mathcal{A}(z) \},\\
  	\mathcal{M}:&=\{z\in\mathcal{R}: Q_\mathcal{A}\subsetneq Q_\mathcal{A}(z) \}=\mathcal{R}\backslash\mathcal{K}.
\end{align}
These sets are better understood by example. Consider the family $T_\mathcal{A}(z)=T\{z,1/2,1/4z\}$ obtained from the stable ternary tree in figure~\ref{upternary} with topological set $Q_\mathcal{A}=\{13\overline{2}\sim 21\overline{2},31\overline{2}\sim 23\overline{2}\}$. This family is defined for the open annulus $\mathcal{R}=\{z\in\mathbb{C}:1/4<|z|<1\}$ since for $|z|\leq1/4$ we would have that $1\leq1/4|z|=|1/4z|=|c_3(z)|$ which is not allowed by definition. A method to obtain points~$z\in\mathcal{M}$ for a given family~$T_{\mathcal{A}}(z)$ consists in imposing extra tip-to-tip equivalence relations~$a\sim b$ with~$a_1\neq b_1$, i.e. figuring out for which parameters~$z\in\mathcal{R}$ the equality $\phi(a)=\phi(b)$ holds. For example, by imposing $133\overline{2}\sim 211\overline{2}$, the equality $0=\phi(133\overline{2})-\phi(211\overline{2})$ expressed in terms of~$\mathcal{A}(z)=\{c_1\rightarrow z, c_2\rightarrow1/2, c_3\rightarrow1/4z\}$ is reduced to 
\begin{align*}
	0=\phi(133\overline{2})-\phi(211\overline{2})&=c_1(1-c_2+c_3)+(c_1c_3^2-c_2c_1^2)/(1-c_2)-c_2\\
	&=-1/4 + 1/8z + z/2 -z^2\\
	&=-(2 z-1) (1 + 4 z^2)/8z.
\end{align*}
Therefore $z=i/2$~and~$z=-i/2$ are in the unstable set~$\mathcal{M}$ since $133\overline{2}\sim 211\overline{2} \notin\{13\overline{2}\sim 21\overline{2},31\overline{2}\sim 23\overline{2}\}$, see the unstable ternary tree in figure~\ref{i2}. For $0=\phi(11\overline{2})-\phi(33\overline{2})$ we obtain the values $z=(-1\pm i \sqrt{3})/4$ which corresponds to the unstable ternary tree with a Sieprinski triangle tipset, see figure~\ref{unstable-sierpinski}. Unstable trees in figure~\ref{unstable3} where obtained by figuring out for which parameters~$z$ the following tip-to-tip conditions hold:
\begin{align*}
\phi(111\overline{2})&=\phi(21\overline{2}), \quad
1 + z + z^2 + 2 z^3=3/2 + z\quad\Longrightarrow\quad 0=1 - 2z^2 - 4 z^3,\\
 \phi(111\overline{2})&=\phi(211\overline{2}), \quad
 1 + z + z^2 + 2 z^3=3/2 + z/2 + z^2\quad\Longrightarrow\quad 0=1 - z - 4z^3,\\
 \phi(1111\overline{2})&=\phi(211\overline{2}) , \quad
 1 + z + z^2 + z^3 + 2 z^4=3/2 + z/2 + z^2\quad\Longrightarrow\quad 0=1 - z -2z^3 - 4 z^4.
\end{align*}

\begin{figure}[H]
\includegraphics[width=6.5in]{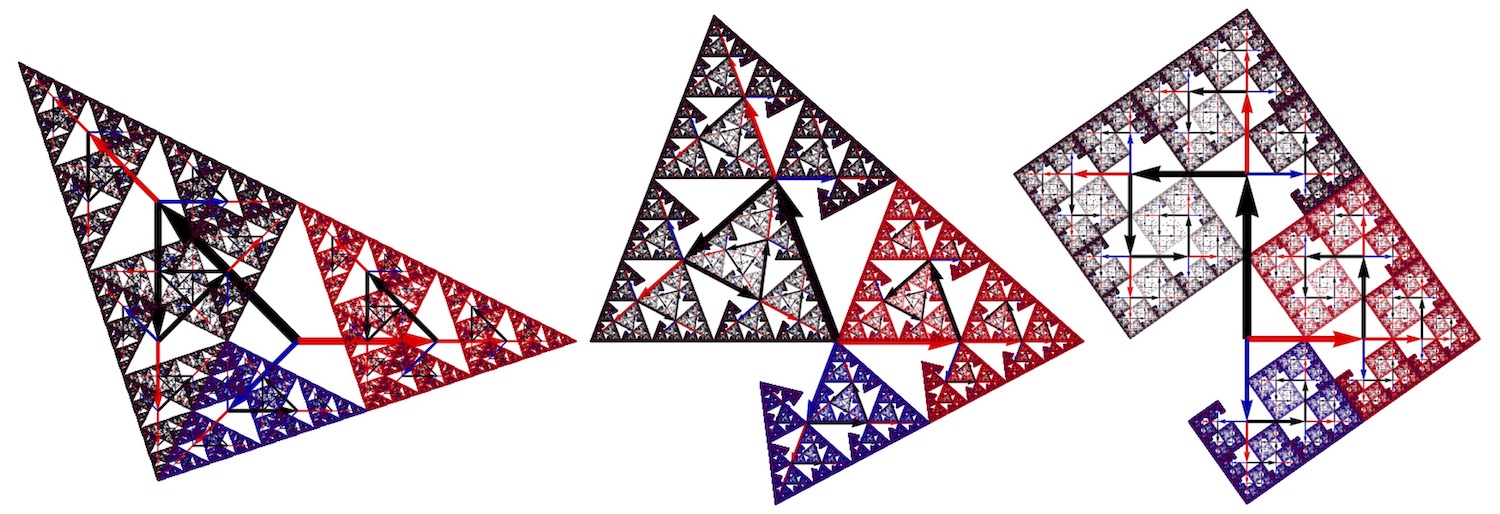}
\centering
\caption{Three examples of unstable ternary trees $T_\mathcal{A}(z)=T\{z,1/2,1/4z\}$, i.e. $Q_\mathcal{A}(z)\neq\{13\overline{2}\sim 21\overline{2},31\overline{2}\sim 23\overline{2}\}$. From right to left their parameters are $z=(-1+i)/2$,  $z=(-1+i\sqrt{7})/4$, and~$z=i/\sqrt{2}$.
}
\label{unstable3}
\end{figure}
\noindent The approximate picture of the unstable set~$\mathcal{M}$ in figure~\ref{region1} was generated by applying this method over all possible equivalence relations $u\overline{2}\sim v\overline{2}$ with $u_1\neq v_1$ and $u,v\in \mathcal{A}(z)^m$ for finite level~$m=10$. The family deduced in~(\ref{z -1/2 1/4z}) from the topological set $Q_\mathcal{A}=\{12\overline{31}\sim22\overline{13},32\overline{13}\sim22\overline{31}\}$ provides another family with a slightly different alphabet $\mathcal{A}(z)=\{z,-1/2,1/4z\}$. The main difference, when compared to $\mathcal{A}(z)=\{z,1/2,1/4z\}$, is that the boundary of its unstable set~$\mathcal{M}$ is predominantly rough, see figure~\ref{region2}. As we prove in~\cite{espigule2019ternary}, for the family $T\{z,1/2,1/4z\}$ this roughness is completely absent and the boundary between the unstable set~$\mathcal{M}$ and the stable set~$\mathcal{K}$ is piece-wise smooth.

\begin{figure}[H]
\begin{overpic}[width=6.4in]{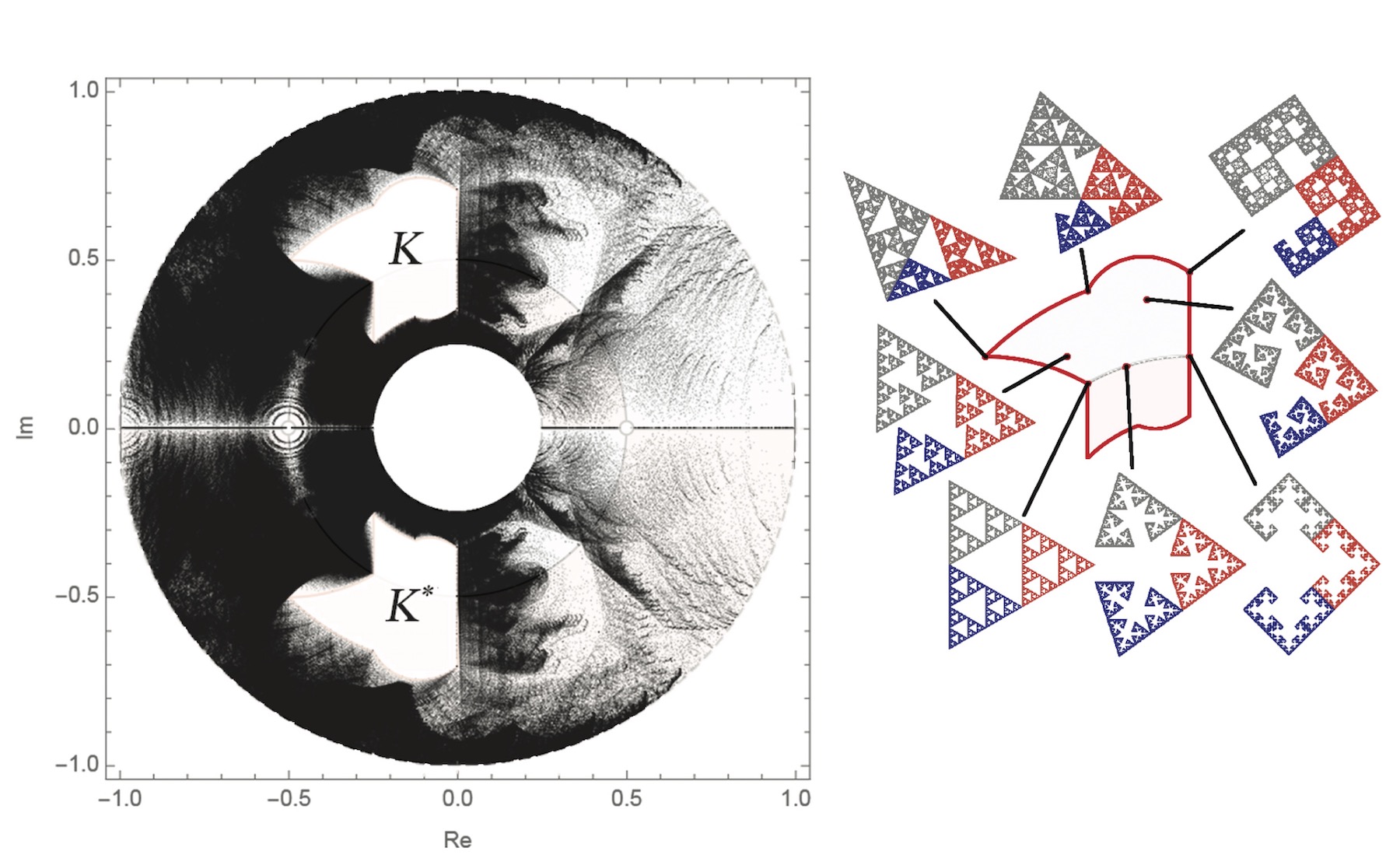}
\put(90,13){$z=i/2$}
\put(76,11){$z=i^{6/5}/2$}
\put(62,12){$z=\frac{-1+i \sqrt{3}}{4}$}
\put(90,56){$z=i/\sqrt{2}$}
\put(74,56){$z=\frac{-1+i\sqrt{7}}{4}$}
\put(61,50){$z=\frac{-1+i}{2}$}
\end{overpic}
\centering
\caption{The unstable set $\mathcal{M}$ for~$T\{z,1/2,1/4z\}$ with a pair of open regions~$K$~and~$K^*$ contained in the stable set~$\mathcal{K}$. On the right, examples of tipsets $F_\mathcal{A}(z)$, five of which have their parameters~$z$ taken from the boundary of~$\mathcal{M}$.
}
\label{region1}
\end{figure}
\begin{figure}[H]
\begin{overpic}[width=6.4in]{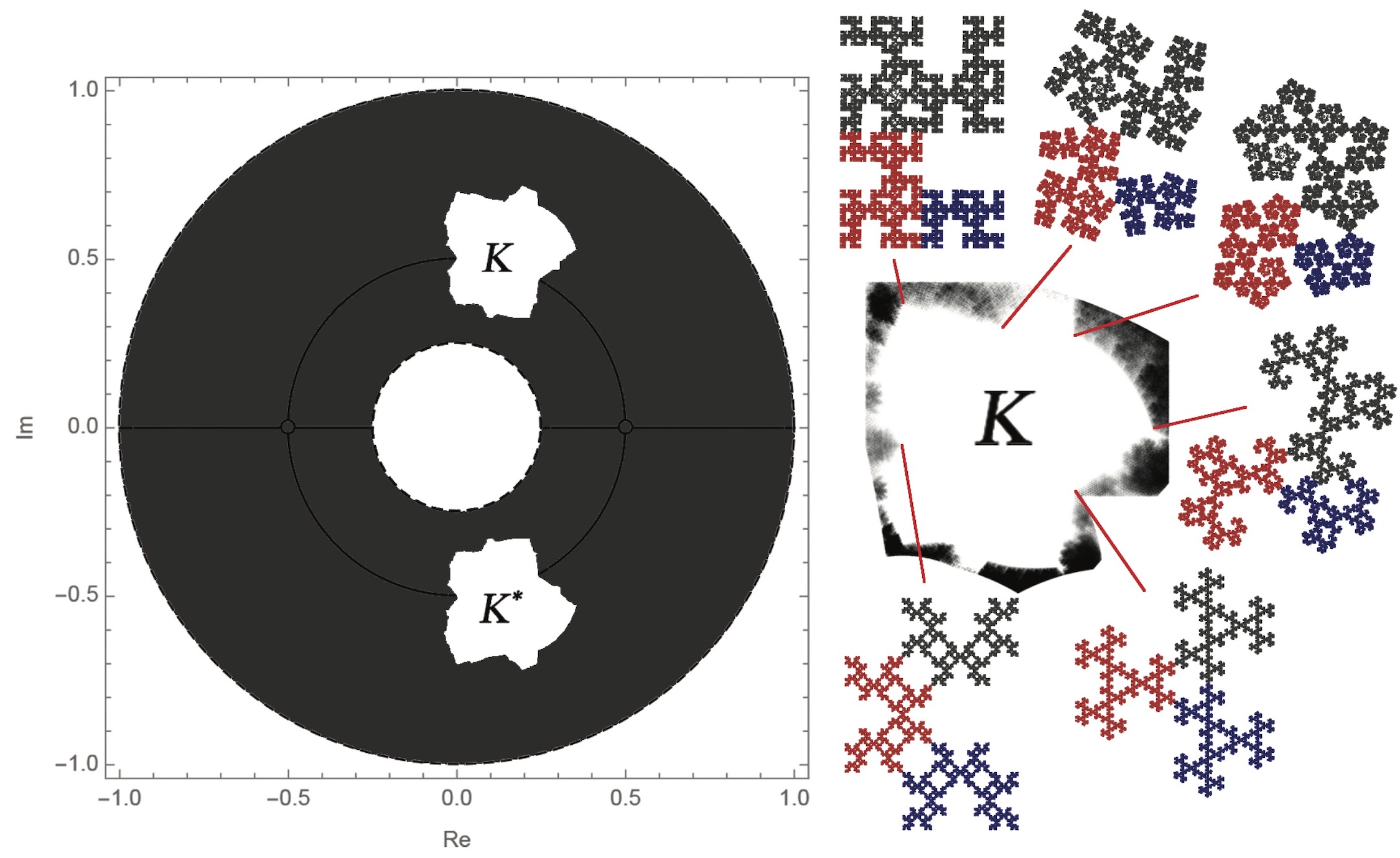}
\put(88,4){$z=\frac{1+i \sqrt{3}}{4}$}
\put(73,4){$z=i/2$}
\put(89,20){$z\approx0.28+0.45i$}
\put(91,57){$z=\frac{1+i\sqrt{7}}{4}$}
\put(74,61){$z\approx0.248+0.660i$}
\put(50,58){$z=i/\sqrt{2}$}
\end{overpic}
\centering
\caption{The unstable set $\mathcal{M}$ for~$T\{z,-1/2,1/4z\}$ with six examples of tipsets $F_\mathcal{A}(z)$ found in~$\partial\mathcal{M}$.
}
\label{region2}
\end{figure}
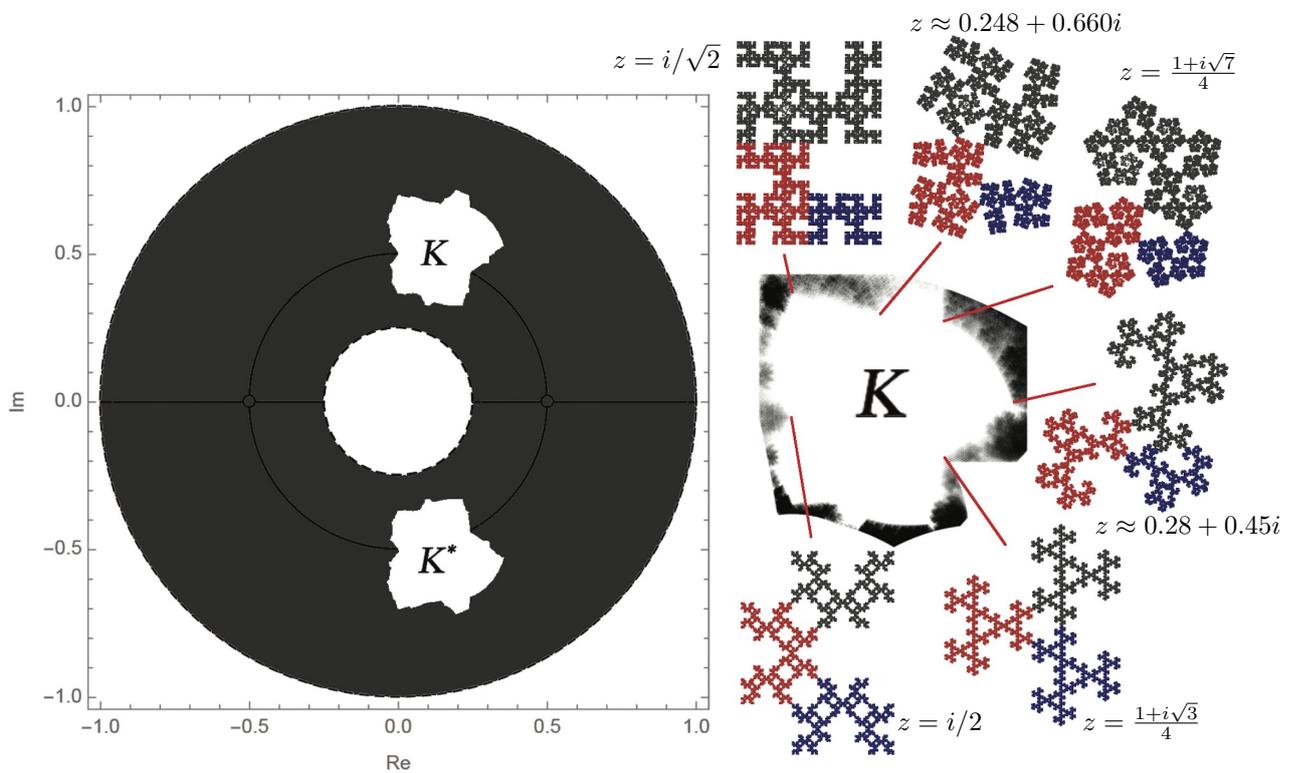

\subsection{The Root Connectivity Set~$\mathcal{M}_0\subseteq\mathcal{M}$}\label{M0}
In this section we will introduce the root connectivity set~$\mathcal{M}_0$ which is a subset of the unstable set~$\mathcal{M}$ that provides another piece of information about the parameter space of our families of study. Given a tipset~$F_\mathcal{A}$, let $O_{F_\mathcal{A}}$ be its \textbf{overlap set} defined as
\begin{equation}\label{O}
	O_{F_\mathcal{A}}:=\bigcup_{j,k\in \mathcal{A}, j\neq k}F_{j\mathcal{A}}\cap F_{k\mathcal{A}},
\end{equation}
and let $I_{F_\mathcal{A}}$ be its \textbf{intersection set} defined as 
\begin{equation}\label{I}
	I_{F_\mathcal{A}}:=F_{1\mathcal{A}}\cap F_{2\mathcal{A}}\cap\dots\cap F_{n\mathcal{A}}.
\end{equation}
If the overlap set is empty, $O_{F_\mathcal{A}}=\emptyset$, we call the tipset $F_\mathcal{A}$~and its tree~$T_\mathcal{A}$ \textbf{totally disconnected}. On the other hand, if the intersection set is not empty,~$I_{F_\mathcal{A}}\neq\emptyset$, we call the tipset $F_\mathcal{A}$~and its tree~$T_\mathcal{A}$ \textbf{totally connected}. Notice that for binary trees we always have that $O_{F_\mathcal{A}}=I_{F_\mathcal{A}}$, hence they are either totally disconnected or totally connected. For~$n\geq3$ this is not necessarily true since it might happen that $I_{F_\mathcal{A}}=\emptyset\neq O_{F_\mathcal{A}}$, see examples in figures~\ref{upternary}-\ref{downternary}. Totally connected trees with their root~$\phi(e_0)=1$ belonging to the intersection set, $\phi(e_0)\in I_{F_\mathcal{A}}$ are special. To distinguish them from other types of totally connected trees, we call such trees~$T_\mathcal{A}$ and their tipsets~$F_\mathcal{A}$ \textbf{root-connected}.

\begin{figure}[H]
\begin{overpic}[width=6in]{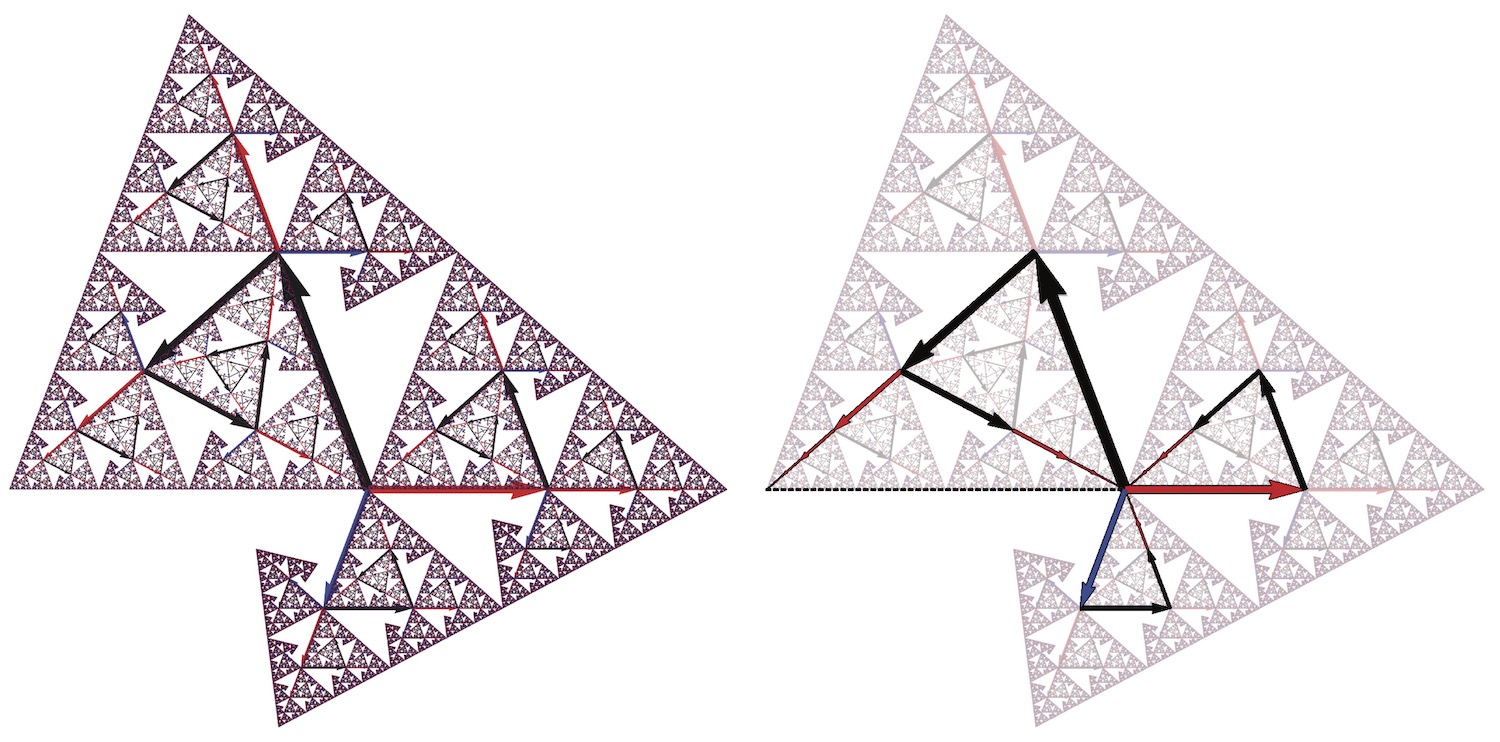}
\put(50,15){$0$}	
\put(64,14.5){$\phi(e_0)=1$}	
\end{overpic}

\centering
\caption{Root-connected ternary tree  $T_\mathcal{A}(z_0)=T_\mathcal{A}((-1+ i\sqrt{7})/4)=T\{(-1+ i\sqrt{7})/4,1/2,(-1- i\sqrt{7})/8\}$. Its~intersection set is $I_{F_\mathcal{A}}=F_{1\mathcal{A}}\cap F_{2\mathcal{A}}\cap F_{2\mathcal{A}}=\{1=\phi(e_0)=\phi(111\overline{2})=\phi(211\overline{2})=\phi(311\overline{2})\}$. The intersection set for this tree is a singleton but the overlap set is not, $I_{F_\mathcal{A}}\neq O_{F_\mathcal{A}}$.
}
\label{rootexample}
\end{figure}

\begin{theo}\label{tiproot connectivity}
\textbf{(Tip-to-root connectivity).} \textit{Let $j\in \mathcal{A}$, $w\in \mathcal{A}^\infty$, and $F_\mathcal{A}$ be a tipset of an $n$-ary complex tree $T_\mathcal{A}$, then the following are equivalent:}
  \begin{enumerate}[label=(\arabic*),ref=(\arabic*)]
   \item $\phi(jw)=\phi(e_0)=1$ \label{tip}
   \item $\phi(w)=0$ \label{zero}
   \item $\phi(jw)=\phi(kw)$ for any  $k\in \mathcal{A}$ \label{tips}
   \item $F_\mathcal{A}$ is root-connected \label{root}
   \item $\phi(v)\in F_\mathcal{A}$ for all $v\in \mathcal{A}^*$ \label{nodes}
  \end{enumerate}
\end{theo}
\begin{proof}
\Implies{tip}{zero}: From \ref{tip} we have that $0=\phi(jw)-\phi(e_0)=1+c_j\phi(w)-1=c_j\phi(w)$ which implies $\phi(w)=0$ since $c_j\neq0$ by definition.
\Implies{zero}{tips}: Independently of the first letter $k\in\mathcal{A}$ we have $\phi(kw)=1+c_k\phi(w)=1=\phi(e_0)=\phi(jw)$.
\Implies{tips}{root}: For each first-level piece there is at least one tip point that meets the root $\phi(1w)=\phi(2w)=\dots=\phi(nw)=1=\phi(e_0)$, hence $I_{F_\mathcal{A}}\neq\emptyset$ and $\phi(e_0)\in I_{F_\mathcal{A}}$ so $F_\mathcal{A}$ is root-connected.
\Implies{root}{nodes}: If $F_\mathcal{A}$ is root-connected we have that $\phi(e_0)\in F_\mathcal{A}$ hence by self-similarity $f_v(\phi(e_0))=\phi(v)\in F_{v\mathcal{A}}\subseteq F_\mathcal{A}$, i.e. all the nodes of the tree $T_\mathcal{A}$ are contained in the tipset $F_\mathcal{A}$. 
\Implies{nodes}{tip}: We have that $\phi(e_0)\in F_\mathcal{A}$ hence $F_\mathcal{A}$ is root-connected and $\phi(e_0)\in F_j$ so there is at least a tip point with word $jw\in \mathcal{A}^\infty$ such that $\phi(jw)=\phi(e_0)=1$.
\end{proof}

\noindent\textit{Remark}. The geometrical explanation behind condition (2) $\phi(w)=0$ in theorem~\ref{tiproot connectivity} comes from the fact that $0$ is actually the base of the \textbf{trunk} of our trees that goes from $0$ to the tree's root $\phi(e_0)=1$. We omit the trunk in our figures but it is always there as a pre-image of first-level branches, i.e. $f_j(\lbrack0,\phi(e_0)\rbrack)=\lbrack\phi(e_0),\phi(j)\rbrack$ for all $j\in \mathcal{A}$. From the structural point of view, $0$ acts as another node and since we have that for a root-connected tree all nodes are contained in the tipset, condition~(5), it shouldn't be surprising that $0$ is also contained in $F_\mathcal{A}$.
\\

\noindent Given an analytic family of trees in one complex variable $T_\mathcal{A}(z)$, we define the \textbf{root connectivity set} $\mathcal{M}_0$ as
\begin{equation}\label{rootmapeq}
	\mathcal{M}_0:=\{z\in\mathcal{M}: \exists w\in \mathcal{A}(z)^\infty \text{ s.t. } \phi(w)=0\}. 
\end{equation}
A direct method to obtain points $z\in\mathcal{M}_0$ consists in finding the zeros of tip points~$\phi(w)$ with letters expressed in terms of the parametric alphabet~$\mathcal{A}(z)$. For example, for $T_\mathcal{A}(z)=T\{z,1/2,1/4z\}$ we have that the tip point $\phi(11\overline{2})=1 + z +  2 z^2$ lies at $0$ when $z_0=(-1+ i\sqrt{7})/4$ and~$z_0=(-1- i\sqrt{7})/4$, see $T_\mathcal{A}((-1+ i\sqrt{7})/4)$ in figures~\ref{unstable3},~\ref{rootexample},~and~\ref{rootmap}. The picture of the unstable set~$\mathcal{M}_0$ shown below was generated using a fast-algorithm covered in~\cite{espigule2019ternary} that checks for a given parameter~$z\in\mathcal{R}$ whether~$z\notin\mathcal{M}_0$.
\begin{figure}[H]
\centering
  \begin{overpic}[width=1\textwidth,tics=10]{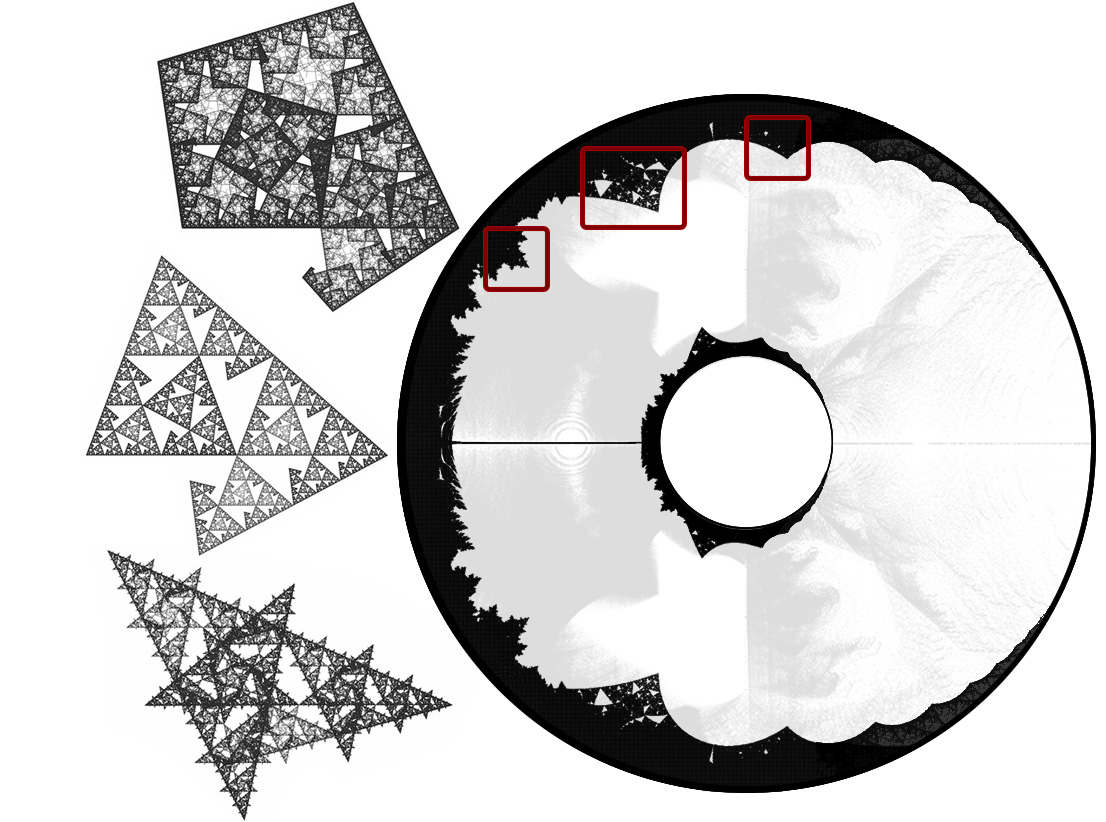}
  \put(4,54){$\phi(111\overline{2})=0$}
  \put(3,30){$\phi(11\overline{2})=0$}
  \put(3,8){$\phi(1111\overline{2})=0$}
  \end{overpic}
  \caption{The root connectivity set $\mathcal{M}_0$ for the family of ternary trees $T\{z,1/2,1/4z\}$ overlaid with the unstable set $\mathcal{M}$ in light gray, see figure~\ref{region1}. On the left, three tipsets $F_\mathcal{A}(z_0)$ with parameters~$z_0$ found in prominent spikes highlighted in $\mathcal{M}_0$. Their tip points meeting at zero are $\phi(111\overline{2})=0$, $\phi(11\overline{2})=0$, and $\phi(1111\overline{2})=0$ with parameter values $z_0\approx0.119492 +i0.813835$, $z_0=(-1+i\sqrt{7})/4$, and $z_0\approx-0.621035+i0.502297$ respectively.
   }
  \label{rootmap}
\end{figure}

\noindent The asymptotic similarity between root-connected tipsets $F_\mathcal{A}(z_0)$ and the boundary regions of~$\mathcal{M}_0$ highlighted in~figure~\ref{rootmap} is very apparent and it is similar in nature to the similarity found between Julia sets of Misiurewicz points and the boundary of the classical Mandelbrot set.
The following theorem tells us that nodes~$\phi(v)$ with associated finite word $v\in \mathcal{A}^*\backslash\{e_0\}$ can be also used to obtain points~$z\in\mathcal{M}_0$.

\begin{theo}\label{noderoot connectivity}
\textbf{(Node-to-root connectivity).} \textit{Let $v\in \mathcal{A}^*\backslash\{e_0\}$, if $\phi(v)=\phi(e_0)$ then $F_\mathcal{A}$ is root-connected.}
\end{theo}
\begin{proof}
 By applying equation~(\ref{vw}) to the tip point of $vvv\dots=\overline{v}\in \mathcal{A}^\infty$ we have that 
 \begin{align*}
 	\phi(\overline{v})&=\phi(v\overline{v})\\
 	&=\phi(v)+v\cdot(\phi(\overline{v})-1)\\
 	&=\phi(e_0)+v\cdot(\phi(\overline{v})-1)\\
 	&=1+v\cdot(\phi(\overline{v})-1),
  \end{align*}
i.e. $\phi(\overline{v})-1=v\cdot(\phi(\overline{v})-1)$. So~$\phi(\overline{v})-1=0$ since the complex product $v=v_1\cdot v_2\cdot\ldots$ is neither~$1$~nor~$0$. Therefore, tip-to-root connectivity $\phi(\overline{v})=1=\phi(e_0)$ takes place and by theorem~\ref{tiproot connectivity} we have that the tipset $F_\mathcal{A}$ is root-connected.
\end{proof}

For each level~$m$ there are~$n^m$ finite words~$v=v_1v_2\dots v_m\in\mathcal{A}^m$ of length~$m$ over the $n$-ary alphabet~$\mathcal{A}$, so for a given one-parameter family of $n$-ary complex trees $T_\mathcal{A}(z)$ it is convenient to define the \textbf{root connectivity set of order}~$m$ as
\begin{equation}
	\mathcal{M}_0(m):=\{z\in\mathcal{M}: \exists v\in \mathcal{A}(z)^m \text{ s.t. } \phi(v)=0\}. 
\end{equation}
The node-to-root connectivity theorem implies that $\mathcal{M}_0(m)$ is contained in~$\mathcal{M}_0$. Moreover, for~$m$~large, the root connectivity set of order~$m$ approximates $\mathcal{M}_0$ since $\lim_{m\to\infty}\mathcal{A}(z)^m=\mathcal{A}(z)^\infty$, i.e. $\lim_{m\to\infty}\mathcal{M}_0(m)=\mathcal{M}_0$. This gives us an algebraic method which is exact but rather slow when compared to our algorithm covered in~\cite{espigule2019ternary} that was used to generate the root connectivity sets~$\mathcal{M}_0$ in figures~\ref{rootmap}~and~\ref{Msetbinary}.

\subsection{Analytic Region $\mathcal{M}_2\subseteq\mathcal{M}$}\label{sec:M2}

The algebraic method stated in section~\ref{M} to obtain points $z$ in the unstable set~$\mathcal{M}$ to generate approximate pictures of~$\mathcal{M}$  is exact but rather slow. Nonetheless there is an analytic method based on a lemma stated below, lemma~\ref{a>2}, that can be applied 
to obtain an analytic region $\mathcal{M}_2\subset\mathcal{M}$.
\\

\noindent Following Jun Kigami's definition of post-critically finite~self-similar structure~\cite{kigami2001analysis}, let the~\textbf{critical set}~$O_\mathcal{A}$ of a complex tree~$T_\mathcal{A}$ be defined as
\begin{align}
	O_\mathcal{A}&:=\{w\in \mathcal{A}^\infty :  \phi(w)\in O_{F_\mathcal{A}}\}, \quad\text{ where $O_{F_\mathcal{A}}$ is the overlap set defined in~(\ref{O}}),
\end{align}
Then we call an $n$-ary complex tree~$T_\mathcal{A}$ \textbf{post-critically finite} (p.c.f.~tree for short) if and only if the \textbf{post critical set}~$P_\mathcal{A}$ defined as
\begin{align}
	P_\mathcal{A}&:=\bigcup_{m\geq1}\sigma^m(O_\mathcal{A}),\label{pcfset}	
\end{align}
is finite, where here $\sigma(w_1w_2w_3\dots):=w_2w_3\dots$ is the one-sided \textbf{shift-map}, and $\sigma^m(O_\mathcal{A})$ is $\sigma$ applied~$m$ times to infinite words $w\in O_\mathcal{A}$. For example, the ternary tree in figure~\ref{upternary} is a p.c.f.~tree since we have that $O_\mathcal{A}=\{13\overline{2},21\overline{2},23\overline{2},31\overline{2}\}$, and therefore $P_\mathcal{A}=\{3\overline{2},\overline{2},1\overline{2}\}$ is finite.
The \textbf{open set condition} (OSC for short), defined by Moran in 1946 and rediscovered by Hutchinson in~\cite{hutchinson1981fractals}, holds if there exists a bounded non empty open set~$U\subset \mathbb{C}$ such that
\begin{equation}\label{osc}
	f_j(U)\subset U \quad \forall j\in \mathcal{A} \quad \text{and}\quad f_j(U)\cap f_k(U)=\emptyset \quad \text{for}\quad j\neq k.
\end{equation}
The OSC has become a standard condition for computing the Hausdorff dimension of self-similar sets. In particular, if a tipset of a complex tree~$T\{c_1,c_2,\dots,c_n\}$ satisfies the open set condition, then its Hausdorff dimension is given by the unique positive number $\alpha=dim_H(F_\mathcal{A})$ that satisfies the following equation
\begin{equation}\label{sim-value}
	|c_1|^\alpha+|c_2|^\alpha+\dots+|c_n|^\alpha=\sum_{j=1}^n |c_j|^\alpha=1.
\end{equation}
If~$T_\mathcal{A}$ is a p.c.f.~tree then the OSC always holds for its tipset~$F_\mathcal{A}$ since the tipset itself is a p.c.f.~self-similar structure, see~\cite{kigami2001analysis}. If the OSC does not hold for a given tree $T\{c_1,c_2,\dots,c_n\}$ then we have that $dim_H(F_\mathcal{A})\neq \alpha$. If that is the case, we call~$\alpha$ the \textbf{similarity dimension} (also known as sim-value~\cite{edgar2007measure}) to distinguish it from the actual Hausdorff dimension. The following lemma shows that~$\alpha$ provides an upper bound for the existence of p.c.f. trees. 

\begin{lemma}\label{a>2}
\textit{Let $F_\mathcal{A}\subset\mathbb{C}$ be a tipset of a tree~$T_\mathcal{A}$. If its similarity dimension exceeds~$ \alpha > 2 $ then $T_\mathcal{A}$ is a non-p.c.f.~tree.}
\end{lemma}
\begin{proof}
	Suppose that $T_\mathcal{A}$ is p.c.f.~and $\alpha>2$ then we would have that $dim_H(F_\mathcal{A})=\alpha>2$ which is a contradiction since $F_\mathcal{A}\subset \mathbb{C}$ and $dim_H(F_\mathcal{A})\leq dim_H(\mathbb{C})=2$.
\end{proof}

\noindent For example, if we apply this lemma to the pair families $T\{z,1/2,1/4z\}$~and~$T\{z,-1/2,1/4z\}$ with unstable sets~$\mathcal{M}$ shown in figures~\ref{region1}~and~\ref{region2} we have 
\begin{equation}
|z|^\alpha+1/2^\alpha+(1/4|z|)^\alpha=1.
\end{equation} 
At $|z|=1/2$, the similarity dimension of $F_\mathcal{A}(z)$ gets its minimum, $\alpha=\log (3)/\log (2)\approx1.585$. For $|z|>1/2$ we have
\begin{equation}
 |z|=\left(2^{-\alpha -1} \left(\sqrt{\left(2^{\alpha }-3\right) \left(2^{\alpha
   }+1\right)}+2^{\alpha }-1\right)\right)^{1/\alpha },
 \end{equation}
 so $|z|$ increases continuously as a function of the similarity dimension that reaches $\alpha=2$ at $|z|=\tau/2$, where $\tau$ is the golden ratio.
 For $|z|<1/2$ we have
 \begin{equation}
 |z|=\left(2^{-\alpha -1} \left(-\sqrt{\left(2^{\alpha }-3\right) \left(2^{\alpha
   }+1\right)}+2^{\alpha }-1\right)\right)^{1/\alpha }
\end{equation}
which decreases continuously as a function of the similarity dimension that reaches $\alpha=2$ at $|z|=1/2\tau$. Therefore, the open annuli regions $\{z:1/4<|z|<1/2\tau\}$~and~$\{z:\tau/2<|z|<1\}$ are contained in the unstable set~$\mathcal{M}$ because by lemma~\ref{a>2} trees in these pair of regions where $\alpha>2$ are non-p.c.f., and consequently unstable. Further details regarding the Hausdorff dimension of connected self-similar sets generated by the families $T\{z,1/2,1/4z\}$~and~$T\{z,-1/2,1/4z\}$ are available at~\cite{espigule2019ternary}. For arbitrary families of complex trees $T_\mathcal{A}(z)=T\{c_1(z),c_2(z),\dots,c_n(z)\}$, the analytic regions~$\mathcal{M}_2\subset\mathcal{M}$ that we are looking for by applying this same technique are simply given by
 \begin{equation}
 \mathcal{M}_2:=\{z\in\mathcal{R} : 	1<|c_1(z)|^2+|c_2(z)|^2+\dots+|c_n(z)|^2\}.
 \end{equation}
 
See below the regions~$\mathcal{M}_2\subset\mathcal{M}$ for the families of binary trees, $T\{z,1+z^2\}$, $T\{z,z+1/(1+z)\}$, and $T\{z,1+z+z^2\}$, obtained in~(\ref{z 1+z^2}),~(\ref{z z+1/(1+z)}), and~(\ref{z 1+z+z^2}) respectively.

 \begin{figure}[H]
 \centering
  \includegraphics[width=6.5in]{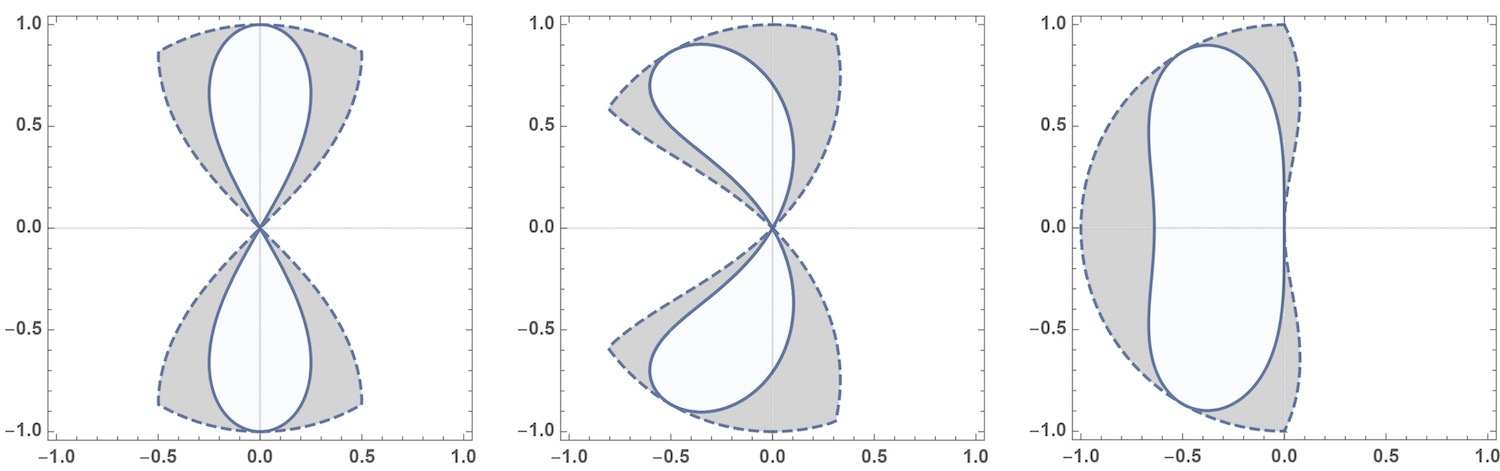}
  \caption{Regions $\mathcal{R}$ shown in fig.~\ref{regionsbinary} overlaid with regions~$\mathcal{M}_2\subset\mathcal{M}$ entirely composed of non-p.c.f.~unstable trees.
    }
  \label{regionsbinary_stable}
\end{figure}

\section{Further Results}\label{sec:further_results}

A tipset~$F_\mathcal{A}$ is a bounded self-similar set in~$\mathbb{C}$, the infinite series ${\phi(w)=\sum_{k=0}^\infty w|_k}$ converges for any $w\in \mathcal{A}^\infty$, and the sequence of partial sums defined by $\phi(w_{|m})=\sum_{k=0}^m w_{|k}$ is a Cauchy sequence that converges to $\phi(w)$. Moreover we have that for $r:=\sup\{|c_1|,|c_2|,\dots,|c_n|\}$, the inequality $|\sum_{k=0}^\infty w_{|k}|\leq\sum_{k=0}^\infty|w_{|k}|\leq\sum_{k=0}^\infty r^k$ holds, and therefore $n$-ary complex trees $T_\mathcal{A}$ are always contained in a disk centered at $\phi(e_0)=1$ with radius $1/(1-r)-1$, i.e. $T_\mathcal{A}\subset D$ where ${D=\{z\in\mathbb{C}:|z-1|\leq 1/(1-r)-1\}}$. A straightforward method to check if a given tipset~$F_\mathcal{A}$ is disconnected consists in showing that the bounding set~$D^k$ of order $k\in\mathbb{N}$ defined as
\begin{equation}
	D^k:=f_1(D^{k-1})\cup f_2(D^{k-1})\cup\dots\cup f_n(D^{k-1}) \text{, where }D^1=f_1(D)\cup f_2(D)\cup\dots\cup f_n(D)
\end{equation}
is disconnected. As soon as we check that for a certain~$k\in\mathbb{N}$ we have that the tipset~$F_\mathcal{A}\subset D^k$ is also disconnected.
The following necessary and sufficient condition for having a connected self-similar set, in our case a tipset~$F_\mathcal{A}$, was already reported in~1992 by Thierry Bousch~\cite{bousch1992quelques}. A tipset $F_\mathcal{A}$ is \textbf{connected} if and only if the graph $G_F$ over the $n$-ary alphabet $\mathcal{A}$ defined as
\begin{equation*}
	(j,k)\in G_F \quad \Longleftrightarrow \quad F_{j\mathcal{A}}\cap F_{k\mathcal{A}}\neq\emptyset \qquad j,k\in\mathcal{A}
\end{equation*}
is connected. Notice that $F_{j\mathcal{A}}\cap F_{k\mathcal{A}}\neq\emptyset$ implies that there is at least a pair of words $a,b\in \mathcal{A}^\infty$ with $a_1=j$ and $b_1=k$ such that $\phi(a)=\phi(b)$.
Therefore $G_F$ can be also defined as
\begin{equation}\label{GF}
	(j,k)\in G_F \quad \Longleftrightarrow \quad \exists a,b\in \mathcal{A}^\infty\quad s.t.\quad \phi(a)=\phi(b)\ ,\ a_1=j \ ,\text{ and }\  b_1=k.
\end{equation}

\subsection{Fractal Dendrites}

A priori if $G_F$ is connected we do not know if $F_\mathcal{A}$ is connected with $\mathbf{int}(F_{j\mathcal{A}}\cap F_{k\mathcal{A}})\neq\emptyset$ for some $j,k\in \mathcal{A}$. A broad class of trees covered so far are trees with a p.c.f.~tipset~$F_\mathcal{A}$ homeomorphic to a \textbf{fractal dendrite}, where by fractal dendrite we mean a self-similar set which is a tree-like topological space, compact, connected and simply connected subset of the complex plane having empty interior. Recall that if $F_\mathcal{A}$ is simply connected it implies that $\mathbb{C}\backslash F_\mathcal{A}$ is also simply connected.

\begin{figure}[H]
\centering
  \includegraphics[width=6.5in]{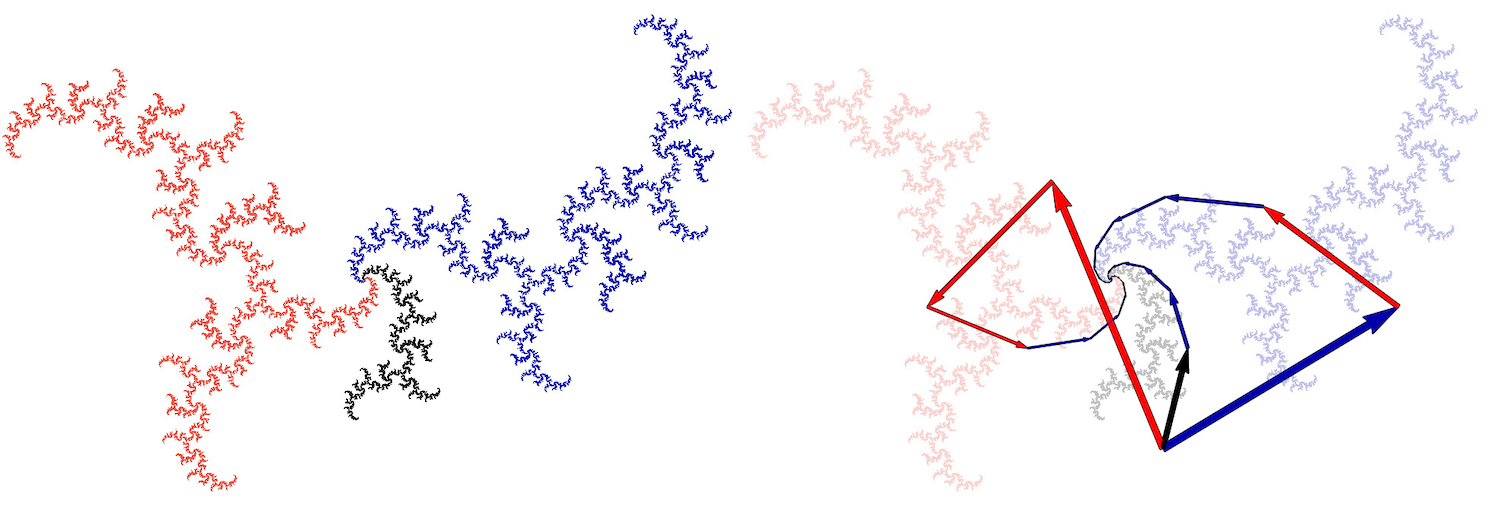}
  \caption{Fractal dendrite, $O_F=I_F=\{x\}$ where $x=\phi(222\overline{3})=\phi(1\overline{3})=\phi(32\overline{3})\approx 0.89+i0.35$, the alphabet is $\mathcal{A}\approx\{0.053 +i0.21,-0.23+i0.566,0.5+i0.3\}$, and~$Q_\mathcal{A}=\{222\overline{3}\sim1\overline{3},1\overline{3}\sim32\overline{3},222\overline{3}\sim32\overline{3}\}$. 
  On the right we have included all branch-paths meeting at~$x$. 
  }
  \label{totally_dendritic}
\end{figure}

\begin{theo}\label{thmdendrite}
\textbf{(Fractal dendrites).} \textit{Let $F_\mathcal{A}$ be connected and $\mathbf{card}(F_j\cap F_k)\leq1$ for all $j,k\in \mathcal{A}$. Then $F_\mathcal{A}$ is a fractal dendrite if and only if $F\backslash (F_j\cap F_k)$ is disjoint for all $j,k\in \mathcal{A}$ such that $F_j\cap F_k\neq\emptyset$.}
\begin{proof}
Observe that if $F_j\cap F_k\neq\emptyset$ then $\mathbf{card}(F_j\cap F_k)=1$, i.e. we have a singleton $\{x_{jk}\}=F_j\cap F_k$. Consequently for any pair of tip points~$\phi(jv)\in F_j\backslash \{x_{jk}\}$~and~$\phi(kw)\in F_k\backslash \{x_{jk}\}$ there is always an arc~$\mathcal{S}\subset F_j\cup F_k$ connecting them such that $x_{jk}\in\mathcal{S}$. If $F\backslash \{x_{jk}\}$ is connected then it means that $x_{jk}$ is not a cutpoint and therefore there exists another arc~$\mathcal{S}'\subset F\backslash \{x_{jk}\}$ connecting $\phi(jv)$~and~$\phi(kw)$. Consequently we have that $F_\mathcal{A}$ is not a dendrite since $\mathcal{S}\cup\mathcal{S}'$ is a closed curve. Now we want to see that if $F_\mathcal{A}$ is a dendrite then $F\backslash \{x_{jk}\}$ is disjoint. Let $\phi(ujv)$~and~$\phi(ukw)$ be any pair of tip points of $F_\mathcal{A}$ with common prefix~$u\in \mathcal{A}^*$, these pair of points are contained in~$F_u$ so by self-similarity we can reframe them into~$F=f_u^{-1}(F_u)$ where now the tip points are~$\phi(jv)=f_u^{-1}(\phi(ujv))$ and~$\phi(kw)=f_u^{-1}(\phi(ukw))$. If $F_j\cap F_k\neq\emptyset$ then~$x_{jk}$ must be a cutpoint separating $\phi(jv)$~and~$\phi(kw)$, consequently $F\backslash \{x_{jk}\}=F\backslash (F_j\cap F_k)$ is disjoint. 
\end{proof}
\end{theo}

\begin{prop}
Let $F_\mathcal{A}$ be totally connected, if the overlap set~$O_{F_\mathcal{A}}$ and the intersection set~$I_{F_\mathcal{A}}$ defined in~(\ref{O})~and~(\ref{I}) are a singleton, $O_{F_\mathcal{A}}=I_{F_\mathcal{A}}=\{x\}$, then $F_\mathcal{A}$ is a fractal dendrite.
\end{prop}
\begin{proof}
For all $j,k\in \mathcal{A}$ the intersections $F_{j\mathcal{A}}\cap F_{k\mathcal{A}}$ take place at the same point $x$, so $F_\mathcal{A}\backslash \{x\}$ is disjoint and by theorem~\ref{thmdendrite} we have that $F_\mathcal{A}$ is a fractal dendrite.
\end{proof}

\subsection{Piece-to-Piece Connectivity}
In order to rescale the intersection of pieces $F_{u\mathcal{A}}=f_u(F_\mathcal{A})$~and~$F_{v\mathcal{A}}=f_v(F_\mathcal{A})$ to the size of the original self-similar set~$F_\mathcal{A}$,  where $u,v\in\mathcal{A}^*$ with $u_1\neq v_1$, Bandt and Graf~\cite{bandt1992self} introduced the concept of a neighbor map $h_{u,v}:=f_u^{-1}f_v$. In the context of complex trees we have from~eq.~(\ref{piece}) that $f_v(z):=\phi(v)+v\cdot(z-1)$ and $f_u^{-1}(z)=(z-\phi(u))/u+1$ so the neighbor map~$h_{u,v}$ can be rewritten as
\begin{align}
	h_{u,v}(z):&=f_u^{-1}(f_v(z))\notag\\
	&=f_u^{-1}(\phi(v)+v\cdot(z-1))\notag\\
	&=(\phi(v)+v\cdot(z-1)-\phi(u))/u+1\notag\\
	&=1+(z-1)v/u+(\phi(v)-\phi(u))/u\label{hmap}
\end{align}
The example below shows how this map applies for a pair of small pieces intersecting tangentially.
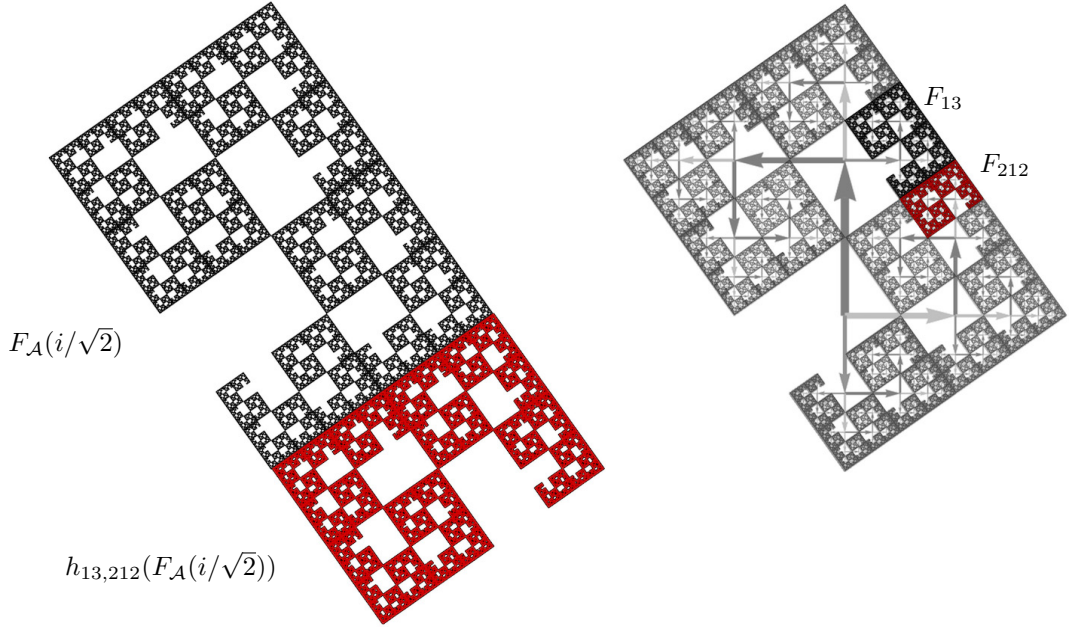
\begin{figure}[H]
\begin{overpic}[width=.9\textwidth,tics=10]{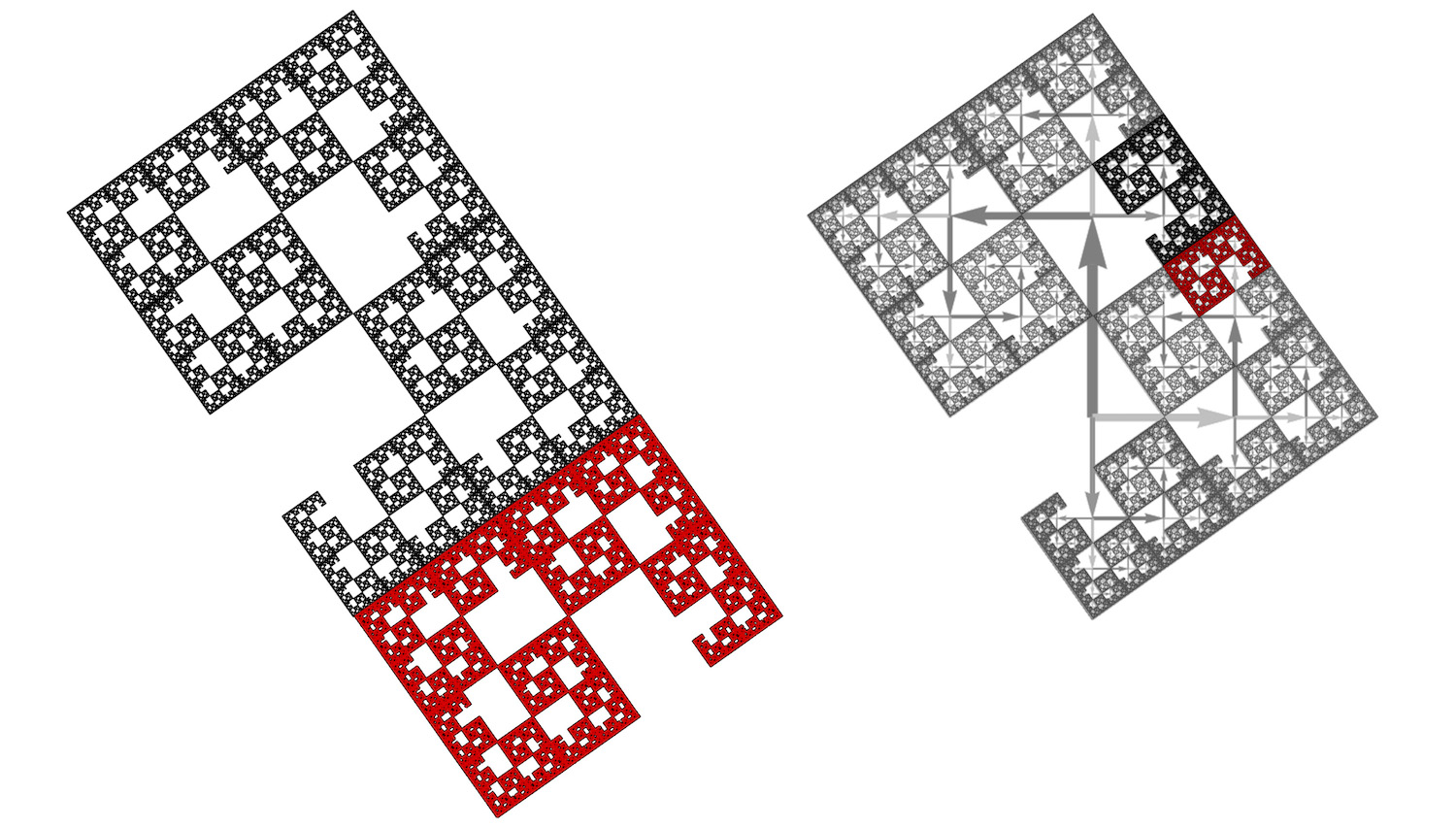}
 \put (1,25) {$F_\mathcal{A}(i/\sqrt{2})$}
 \put (6,5) {$h_{13,212}(F_\mathcal{A}(i/\sqrt{2}))$}
 \put (82,47) {$F_{13}$}
 \put (87,41) {$F_{212}$}
\end{overpic}
\centering
\caption{The neighbor set~$h_{13,212}(F_\mathcal{A}(i/\sqrt{2}))$ intersects the tipset of the unstable tree~$T_\mathcal{A}(i/\sqrt{2})$ in figure~\ref{unstable3}, the same way as piece $F_{13}$ intersects piece~$F_{212}$.
}
\label{pieces}
\end{figure}


\noindent The following theorem is a result of having $h_{u,v}=id$. It applies to nodes $\phi(u)$~and~$\phi(v)$ encoded by words~$u$~and~$v$, not necessarily of same length, such that their complex product $u_1\cdot~u_2\cdot~u_3\cdot\ldots=v_1\cdot~v_2\cdot~v_3\cdot\ldots$ is the same.

\begin{theo}\label{Piece-to-piece connectivity}
\textbf{(Piece-to-piece connectivity).} \textit{Let $u:=u_1u_2\dots u_{l-1} u_l$ and $v:=v_1v_2\dots v_{m-1}v_m$ be a pair of finite words with~$u_1\neq v_1$ and length $l$~and~$m$ respectively, then the following are equivalent:}
  \begin{enumerate}[label=(\arabic*),ref=(\arabic*)]
   \item $F_{u\mathcal{A}}=F_{v\mathcal{A}}$ \label{identity}
   \item $h_{u,v}=id$ \label{hid}
   \item $\phi(u)=\phi(v)$~and~$\phi(u')=\phi(v')$ where $u':=u_1u_2\dots u_{l-1}$ and $v':=v_1v_2\dots v_{m-1}$ \label{uv}
  \end{enumerate}
\end{theo}
\begin{proof}

\ref{identity}$\implies$\ref{hid} since $F_{u\mathcal{A}}=f_u(F_\mathcal{A})$~and~$F_{v\mathcal{A}}=f_v(F_\mathcal{A})$ we have that $f_u(F_\mathcal{A})=f_v(F_\mathcal{A})$ so $f_u^{-1}f_v(F_\mathcal{A})=F_\mathcal{A}$, i.e.~$h_{u,v}=id$. \ref{hid}$\implies$\ref{uv} from~eq.~(\ref{hmap}) this is immediate. \ref{uv}$\implies$\ref{identity} subtracting $\phi(u')=\phi(v')$~from~$\phi(u)=\phi(v)$, we get that $u=\phi(u)-\phi(u')=\phi(v)-\phi(v')=v$, i.e. the complex product of the individual letters is the same,~$u_1\cdot u_2\cdot u_3\dots =v_1\cdot v_2\cdot v_3\dots$. So from~eq.~(\ref{piece}), $F_{u\mathcal{A}}=\phi(u)+u\cdot(F_\mathcal{A}-1)=\phi(v)+v\cdot(F_\mathcal{A}-1)=F_{v\mathcal{A}}$.
\end{proof}

\begin{figure}[H]
\begin{overpic}[width=1\textwidth,tics=10]{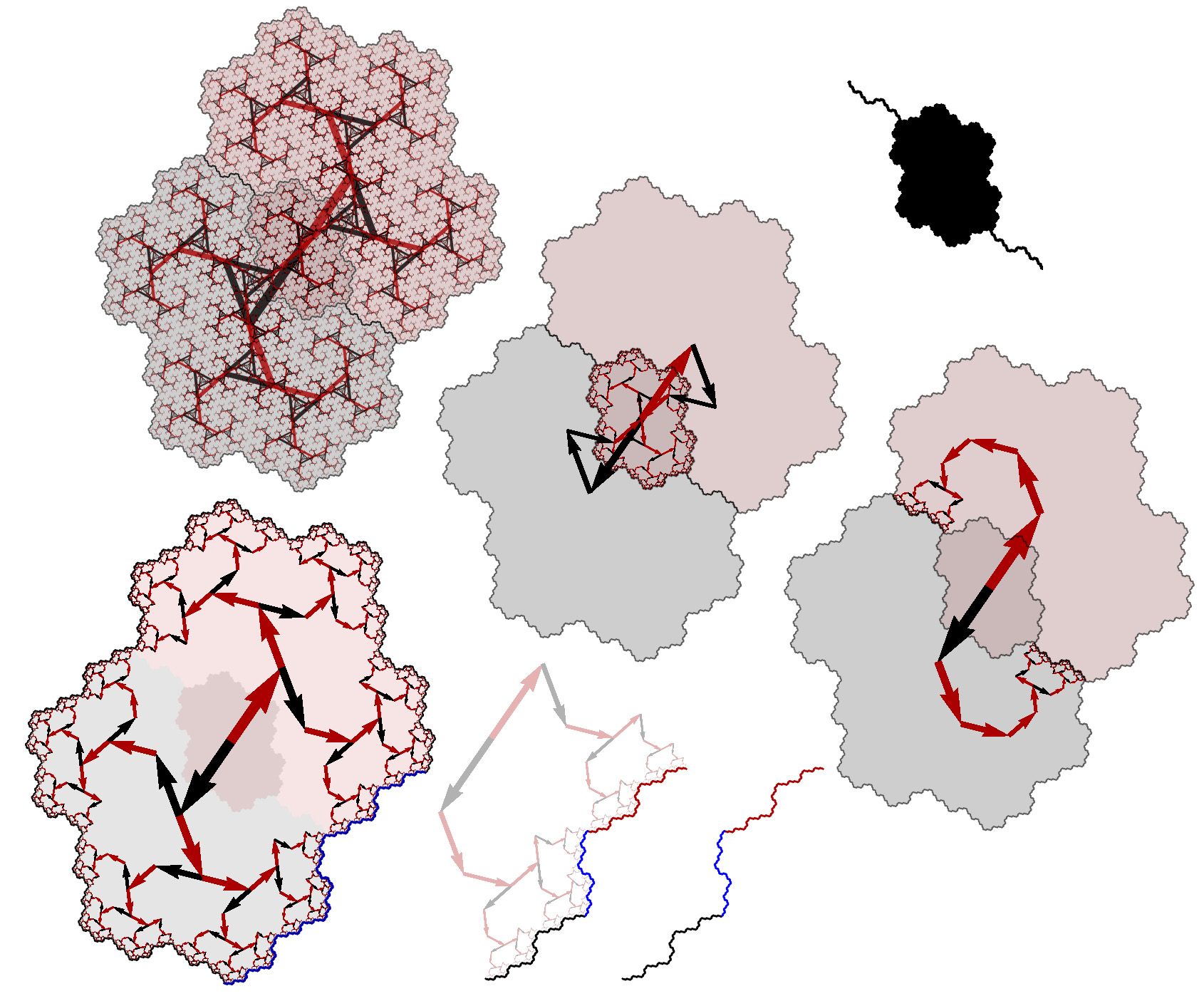}
 \put (62,7) {$L:=f_{212}(L)\cup f_{1222}(L)\cup f_{122}(L)$}
 \put (54,16) {$f_{212}(L)$}
 \put (50,9) {$f_{1222}(L)$}
 \put (45,3) {$f_{122}(L)$}
 \put (70,72) {$L_1$}
 \put (74,61) {$X$}
 \put (82,60) {$L_2$}
  \put (79,74) {$F_1\cap F_2$}
\end{overpic}
\centering
\caption{Binary complex tree $T\{c,-c\}$ where $c\approx 0.737e^{-2.176i}$ is a root of $1+x+x^2-x^3$. Since we have the following tip-to-tip intersections $\phi(11122)=\phi(21121)$, $\phi(11121)=\phi(21122)$, and $\phi(1112)=\phi(2112)$ theorem~6 tells us that the following piece-to-piece exact overlaps take place $F_{11122}=F_{21121}$, $F_{11121}=F_{21122}$. Moreover $F_{1112}=F_{11121}\cup F_{11122}=F_{21121}\cup F_{21122}=F_{2112}$, i.e.~$F_{1112}=F_{2112}$. For this particular example the tipset $F\{c,-c\}$ is the classical Rauzy fractal shown in figure~\ref{treebonacci} with the particularity of having the first-level pieces $F_1$~and~$F_2$ intersecting in an non-empty set $F_1\cap F_2=L_1\cup X\cup L_2$ where $X=F_{1112}=F_{2112}$ and $L_1$ and $L_2$ are given in terms of the fractal curve $L:=f_{212}(L)\cup f_{1222}(L)\cup f_{122}(L)$ as $L_1=f_{22222}(L)$, $L_2=f_{12222}(L)$. 
}
\label{piece1110-0110}
\end{figure}
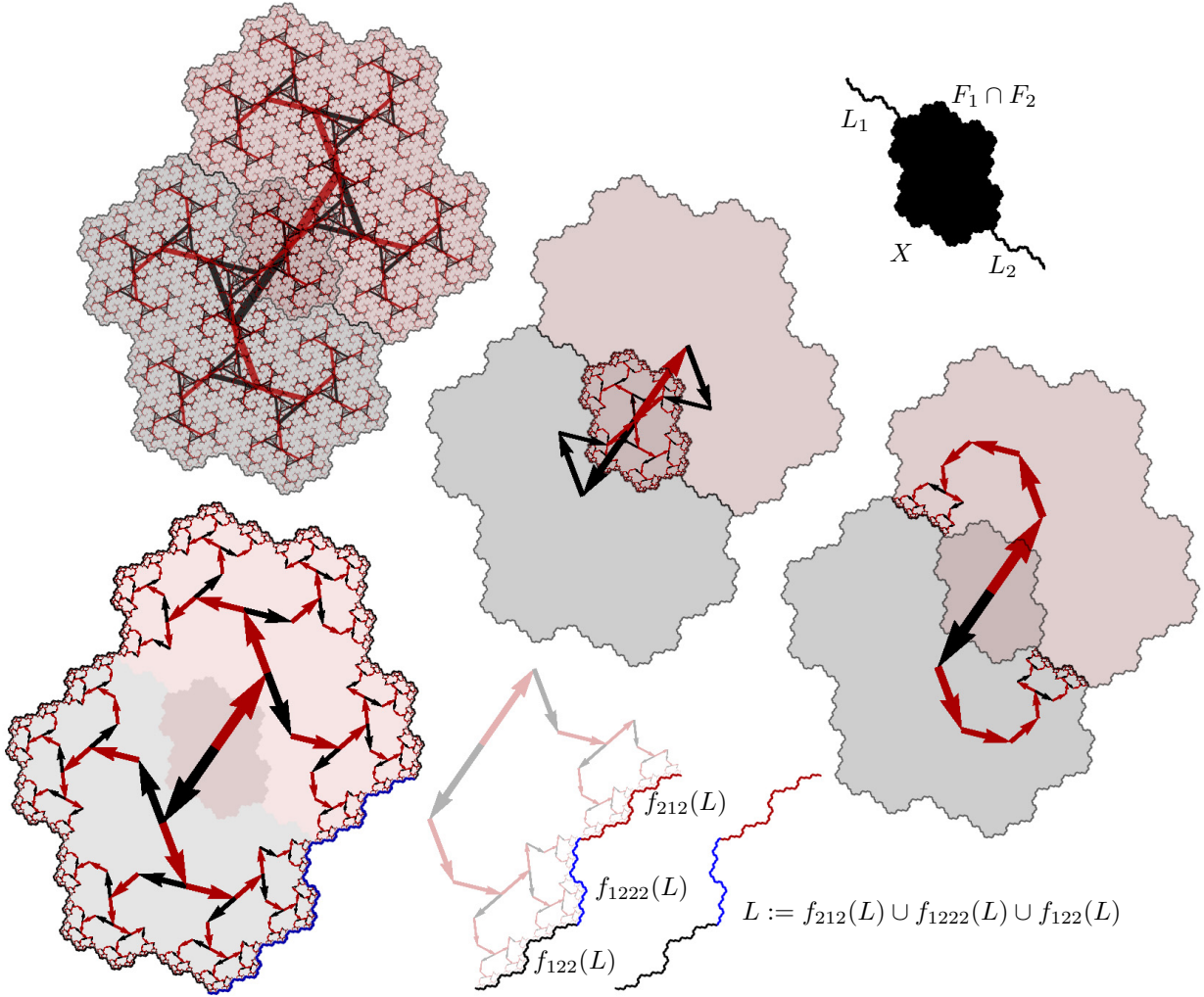

\section{Related Work}\label{sec: related}

The notion of complex tree introduced here is related to a variety of problems that arise naturally in different fields. For example, the cover picture in Bill Thurston's last paper, \textit{Entropy in dimension one}~\cite{thurston2014entropy}, see also the annotated version with a section of additional remarks added by John Milnor~\cite{bonifant2014frontiers}, shows a set that fits, when restricted to the unit disk, into the root connectivity set~$\mathcal{M}_0$ for the family of binary complex trees $T_\mathcal{A}(c):=T_{\{c,-c\}}$, see figure~\ref{z-zMvsM0}. 
\subsection{Family of Binary Complex Trees $T\{c, -c\}$}\label{z-z}
The set $\mathcal{M}_0$ is the set of roots of all polynomials with coefficients~$\pm1$ which are contained in the unit disk. The connection between $\mathcal{M}_0$ and the set investigated by Thurston was presented in a paper by Tiozzo~\cite{tiozzo2016continuity}. If we start with any point $z\in\mathbb{C}$, the limits of all sequences of compositions of maps $f_1(z)=1+cz$~and~$f_2(z)=1-cz$ allow us to express $F_{\{c,-c\}}$ as the set of power series $1\pm c\pm c^2\pm c^3\pm\dots$. The problem of determining for which parameters~$c$ the tipset $F_{\{c,-c\}}$ is disconnected, i.e.~$Q_{\{c,-c\}}=\emptyset$, has emerged over the last 30~years as an area of interest. As it was nicely put by Calegari et al in a paper recently published in \textit{Ergodic Theory and Dynamical Systems}~\cite{calegari2016roots}, ``The richness and mathematical depth of these various sets has barely begun to be plumbed". The first systematic study about the connectivity of these fractals was carried out by Barnsley and Harrington~in~1985~\cite{barnsley1985mandelbrot}. They defined the set $\mathcal{M}$ shown in figure~\ref{z-zMvsM0}, they proved that the boundary of~$\mathcal{M}$ is contained in~$\{c:1/2\leq|c|\leq1/\sqrt{2}\}$, and among several other observations they noted apparent holes near the boundary~$\partial\mathcal{M}$. This seminal paper was followed by many other results, below we cite some of these results but our list is incomplete. Different authors use different pairs of mappings $f_1(z)$ and $f_2(z)$, we use the one introduced by Baez~\cite{baez2009beauty} that corresponds to our maps associated to complex trees, eq.~(\ref{ifs}) for $c_1=c$~and~$c_2=-c$, i.e.~$f_1(z)=1+cz$~and~$f_2(z)=1-cz$.

\begin{figure}[H]
\centering
  \includegraphics[width=1\textwidth]{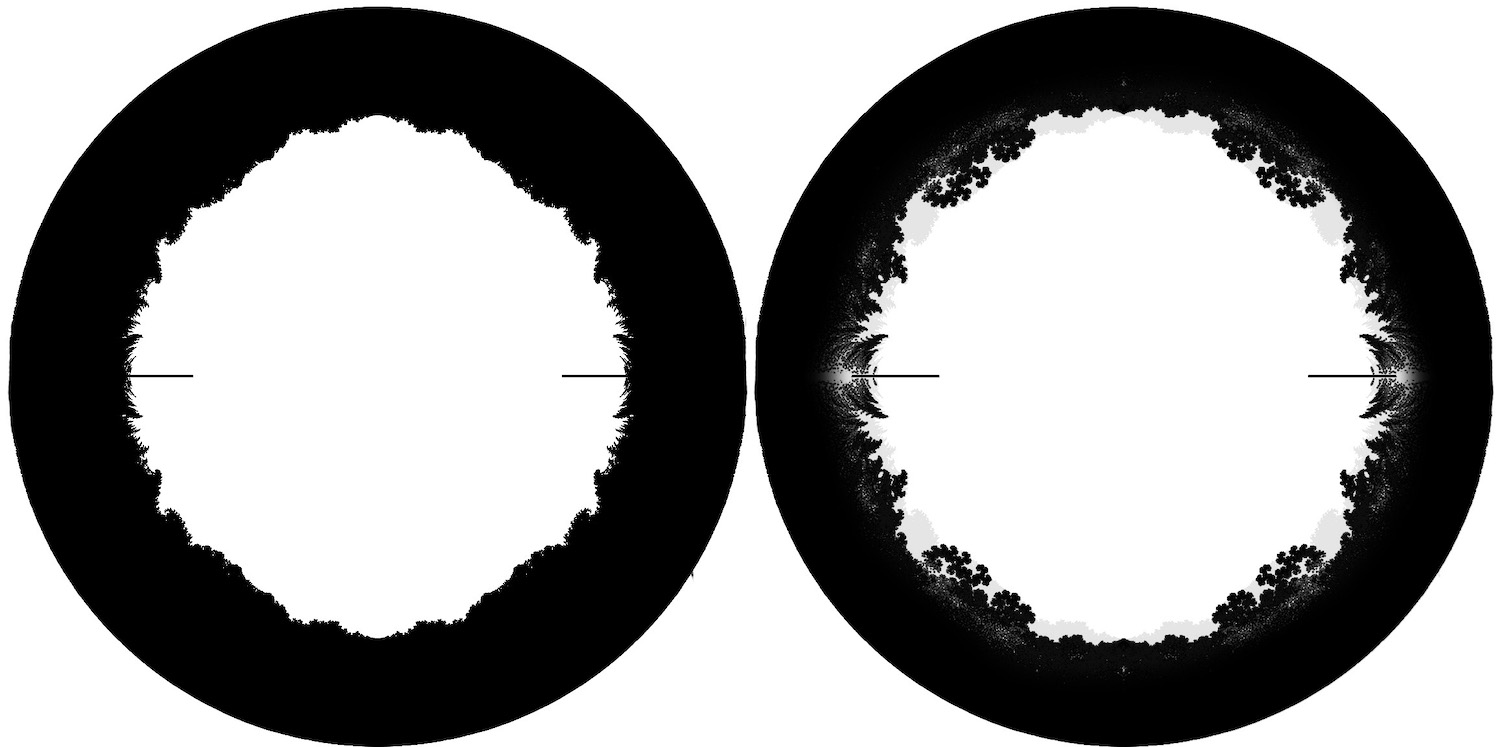}
    \caption{The unstable set $\mathcal{M}=\{c:Q_{\{c,-c\}}\neq\emptyset\}$, and the root connectivity set $\mathcal{M}_0=\{c:0\in F_{\{c,-c\}}\}$.
    }
  \label{z-zMvsM0}
\end{figure}

\noindent In the late 80s Thierry Bousch proved that $\mathcal{M}$~and~$\mathcal{M}_0$ are both connected and locally connected~\cite{bousch1993connexite}. Recall that in 1982 Douady and Hubbard proved that the classical Mandelbrot set is connected but their conjecture about its local connectivity is still open. 
In 2002 Bandt 
presented a fast algorithm for generating high resolution pictures of $\mathcal{M}$ that revealed many interesting features~\cite{bandt2002mandelbrot}. He showed the existence of certain algebraic points in~$\partial\mathcal{M}$ that he called \textit{landmark points}, he proved the existence of a hole in~$\mathcal{M}$, and he conjectured that the interior of~$\mathcal{M}$ is dense away from the real axis. In relation to Bandt's conjecture, Solomyak and Xu~\cite{solomyak2003themandelbrot} proved that the interior is dense in a neighborhood of the imaginary axis. In 2016 Bandt's conjecture was finally proved by Calegari et al~\cite{calegari2016roots} using a new technique that they introduced to construct and certify interior points of~$\mathcal{M}$. They also proved the existence of infinitely many holes in $\mathcal{M}\backslash\mathcal{M}_2$.

\subsection{Family of Symmetric Binary Complex Trees $T\{c, c*\}$}

More or less at the same time when the connectivity locus of self-similar sets~$F_{\{c,-c\}}$ was introduced, Douglas Hardin, at that time a PhD student of Michael Barnsley, studied a closely related family of self-similar sets that corresponds to binary complex trees $T_\mathcal{A}(c):=T_{\{c,c^*\}}$, where $c^*$ is the complex conjugate of~$c$. Among other things, Hardin and Barnsley proved that the boundary of the unstable set~$\mathcal{M}=\{c:Q_{\{c,c^*\}}\neq\emptyset\}$ is piece-wise smooth~\cite{barnsley1989mandelbrot}. This set was explored with improved figures in chapter~8 of Barnsley's book \textit{Fractals Everywhere}. 
The first picture of the root connectivity set~$\mathcal{M}_0=\{c:0\in F_{\{c,c^*\}}\}$ is found in Stephen Wolfram's book \textit{A New Kind of Science}~\cite{wolfram}. Wolfram considered the sets~$W_\lambda:=\{c:\lambda\in F_{\{c,c^*\}}\}$ for different values~$\lambda$ including $\lambda=0$ which corresponds to~$\mathcal{M}_0=W_0$, see figure~\ref{Msetbinary}. Notice that because of the tip-to-root connectivity theorem, we have that for~$\lambda=1=\phi(e_0)$ the resulting \textit{Wolfram set}~$W_\lambda$ is identical to the one obtained for~$\lambda=0$, i.e.~$W_1=W_0=\mathcal{M}_0$.
\begin{figure}[H]
\centering
  \includegraphics[width=1\textwidth]{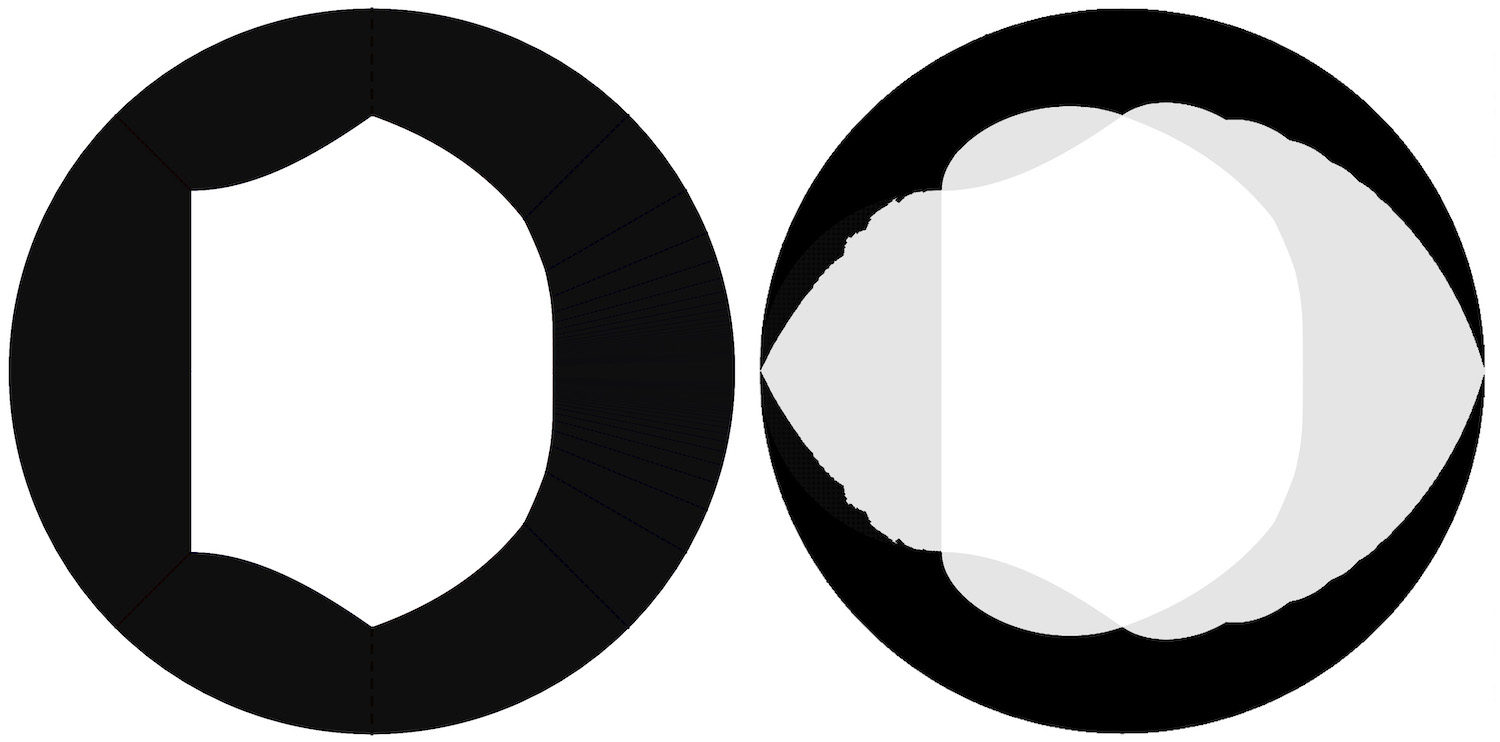}
    \caption{The unstable set $\mathcal{M}=\{c:Q_{\{c,c^*\}}\neq\emptyset\}$, and the root connectivity set $\mathcal{M}_0=\{c:0\in F_{\{c,c^*\}}\}$.
    }
\label{Msetbinary}
\end{figure}

\noindent Mandelbrot and Frame~\cite{mandelbrot1999canopy} studied this family of mirror-symmetric binary trees~$T_{\{c,c^*\}}$ in a purely geometric way, describing the algebraic equations that define $\partial\mathcal{M}$ in terms of~$r(\theta)$, i.e. the scaling ratio $r=|c|$ as a function of the angle $2\theta$ between branches. Tara~Taylor expanded their results by considering several topological properties~\cite{taylor2005computational}~\cite{taylor2007homeomorphism}~\cite{taylor2009topological}~\cite{taylor2009new}. Deheuvels developed analysis on them~\cite{deheuvels2016sobolev}. 
And Pagon~\cite{pagon2003self} suggested that it might be possible to generalize this family of binary trees to $n$-ary trees. In 2013, motivated by this suggestion, the author presented nine families of equations~$r(\theta)$ that parameterize all $n$-ary trees with equally spaced mother branches of same length~$r(\theta)$, symmetric around the real axis, and with tipset \textit{just-touching}~\cite{espigule2013generalized}. When represented in polar coordinates, equations~$r(\theta)$ describe the boundary of unstable sets~$\mathcal{M}$ for each family of mirror-symmetric $n$-ary trees, see figure~\ref{n-aryMset}. The limiting elements of these curves~ $r(\theta)$ generate tipsets with $n$-fold rotational symmetry that are Sierpi\'nski gaskets. These gaskets also appear as special elements of one-parameter families of self-similar sets known as $n$-gon fractals that were first studied by Bandt~and~Hung~\cite{bandt2008fractal}. 
For $n=2$ the family corresponds to~$T_{\{c,-c\}}$ considered in section~\ref{z-z}. For~$n\geq2$ these families of $n$-ary complex trees in one complex variable~$c$ are defined as
\begin{equation}
T_\mathcal{A}(c):=T\{ce^{i2\pi(1/n)},ce^{i2\pi(2/n)},\dots,ce^{i2\pi((n-1)/n)},c\}
\end{equation}
Recent results about $n$-gon fractals include the proof by Himeki and Ishi showing that the unstable set~$\mathcal{M}$ for $n=4$ is regular-closed~\cite{himeki2018mathcal}, and a technique by Calegari~and~Walker~\cite{calegari2018extreme} to obtain extreme points of tipsets for arbitrary~$c$.

\begin{figure}[H]
\centering
  \includegraphics[width=1\textwidth]{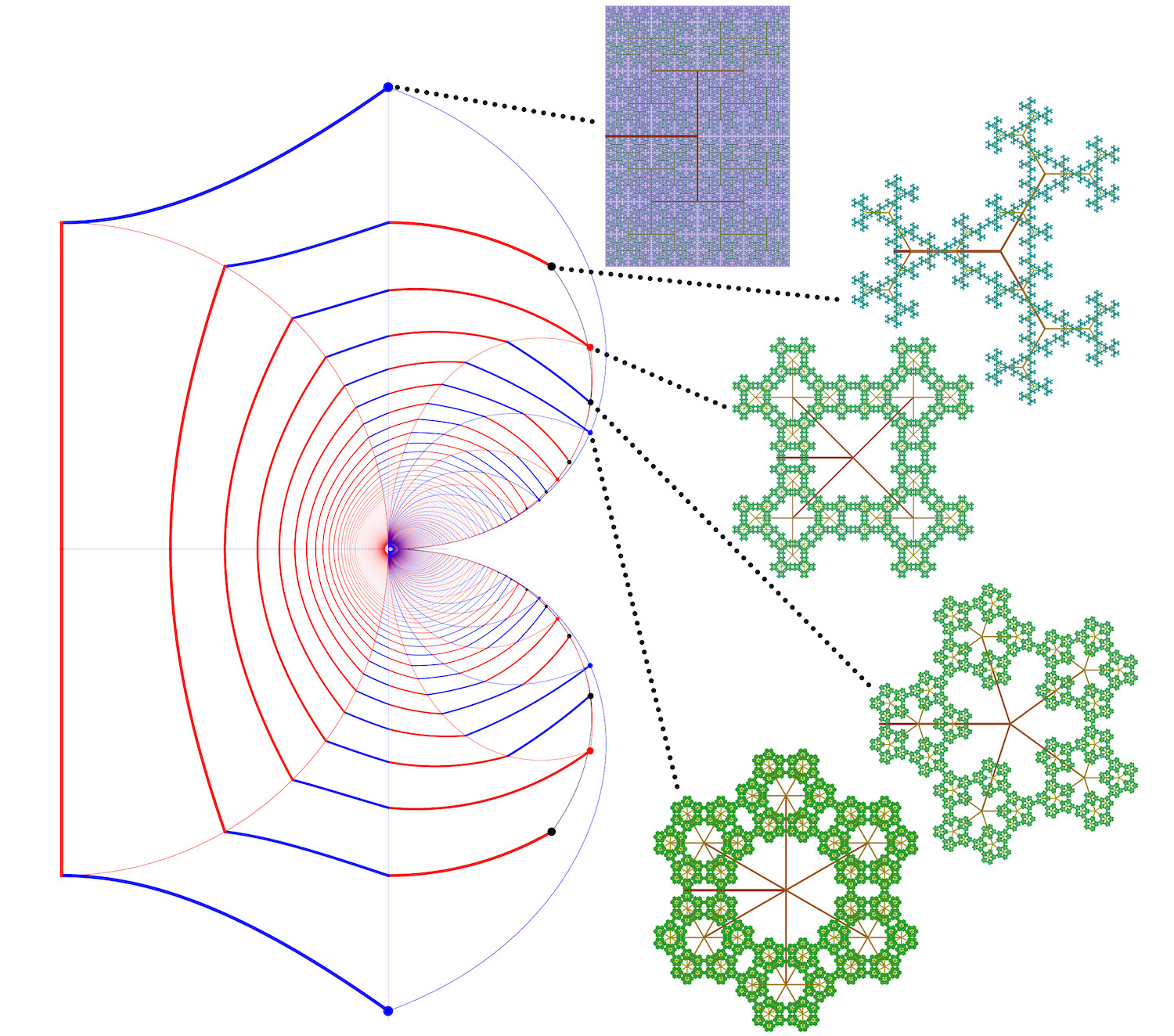}
    \caption{The boundaries of the unstable sets~$\mathcal{M}$ for down $n$-ary symmetric fractal trees~\protect\cite{espigule2013generalized}. On the right, Sierpi\'nski gasket $n$-ary~trees for~$n=2,3,4,5,\text{ and }6$.    
    }
  \label{n-aryMset}
\end{figure} 


\subsection{Polygonal Dendrites}

Normally, tipsets~$F_\mathcal{A}$ of stable complex trees~$T_\mathcal{A}$ are fractal dendrites, but there are exceptions like the examples in figure~\ref{stablegasket} which are not. Nonetheless, self-similar fractal dendrites are topological objects with some unique properties that make them special. Among the possible types of dendrites there is a class known as polygonal dendrites that has been considered recently by Andrey Tetenov, Mary Samuel, and their collaborators~\cite{dendrites2018}. The idea of $\delta$-deformations of polygonal dendrites introduced in~\cite{drozdov2018delta} is being captured by our notion of stable set~$\mathcal{K}$. Many polygonal dendrites can be associated to stable complex trees, for instance, $\delta$-deformations also apply to tipsets of binary complex trees from the families obtained in~(\ref{z 1+z^2}),~(\ref{z z+1/(1+z)}),~and~(\ref{z 1+z+z^2}) as long as the tipset~$F_\mathcal{A}(z)$ is a structurally stable dendrite,~i.e.~$z\in\mathcal{K}$, see figure~\ref{regionsbinary}.

\subsection{Rauzy Fractals}

The ternary tree with alphabet $\mathcal{A}=\{c_1,c_2,c_3\}\approx\{-0.191+0.509 i,-0.420-0.606 i,0.389-0.097 i\}$ is root-connected and unstable because $\phi(2\overline{21})=0$ and $Q_{\{c_1,c_2,c_3\}}=\{12\overline{3}\sim31\overline{3},11\overline{3}\sim22\overline{3},12\overline{21}\sim22\overline{21},\dots\}$ leads to a finite set of numerical solutions which includes~$\mathcal{A}=\{c_1,c_2,c_3\}$. Also notice that the equivalence relation $12\overline{21}\sim22\overline{21}$ does not involve the symbol~3 so it admits a tipset connected binary tree as a subset. Figure \ref{treebonacci} shows how the binary tree $T\{c_1,c_2\}\approx T\{-0.191+0.509 i,-0.420-0.606 i\}$ and its tipset are subsets of the ternary tree~$T\{c_1,c_2,c_3\}$.
The complex-values of the letters $c_1,c_2,c_3$ are Galois conjugates of Pisot numbers: $c_1$ is a root of $-1-x-3x^2+x^3$, $c_2$ is a root of $-1-x-x^2+x^3$, and $c_3$ is a root of $-1+5x-7x^2+x^3$. For the author's surprise, the relation between these two self-similar sets was reported recently in~\cite{sirvent2015fibonacci}. The self-similar set associated to the unstable ternary tree is the classical Rauzy fractal, and the fractal associated to the unstable binary tree is called the Fibonacci-gasket.

\begin{figure}[H]
  \includegraphics[width=4.5in]{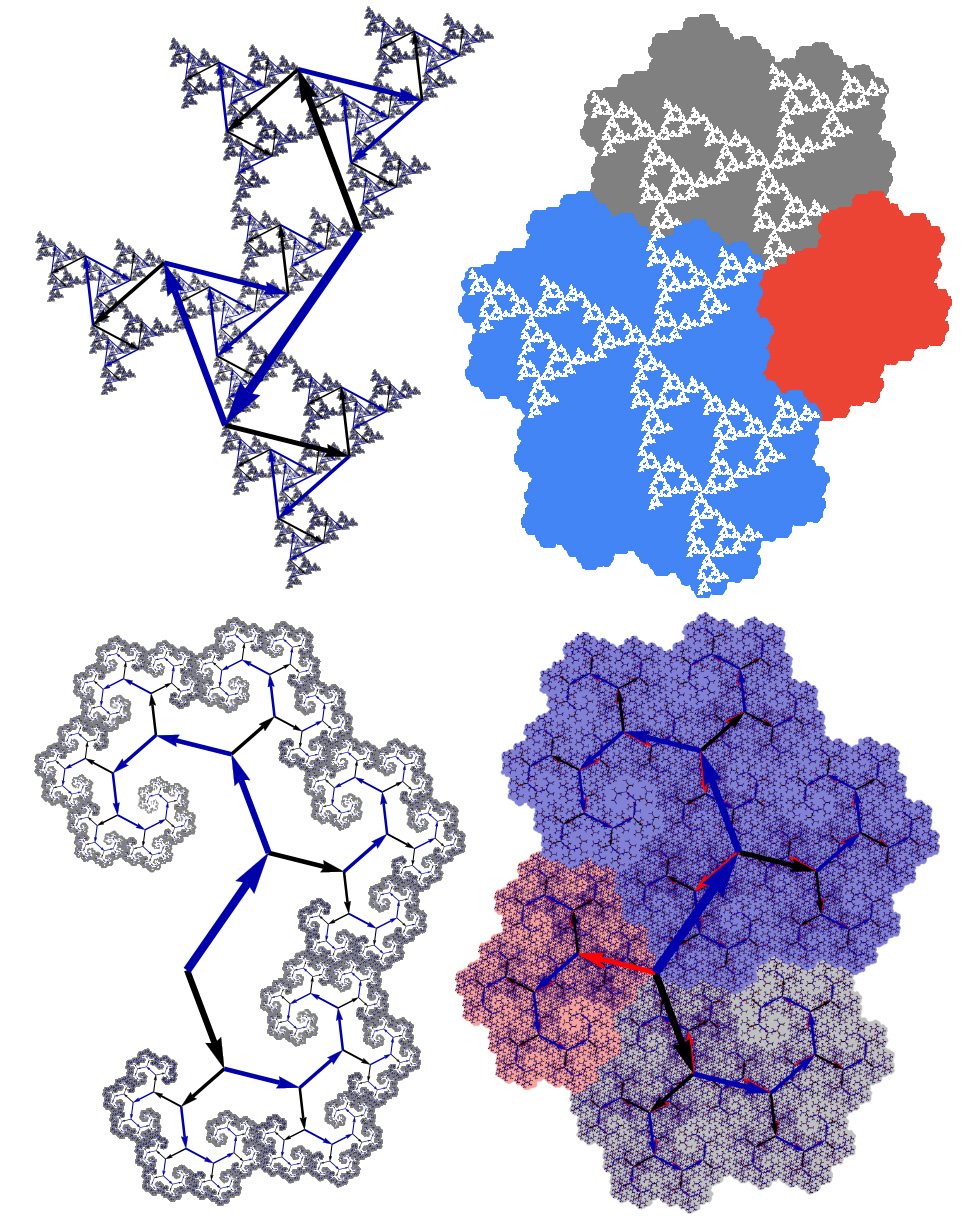}
  \centering
  \caption{Unstable binary tree $T\{c_1,c_2\}\approx T\{-0.191+0.509 i,-0.420-0.606 i\}$ which is a subset of unstable ternary tree $T\{c_1,c_2,c_3\}\approx T\{-0.191+0.509 i,-0.420-0.606 i,0.389-0.097 i\}$.   Their tipset are known as the Fibonacci-gasket and the classical Rauzy fractal, see~\protect\cite{sirvent2015fibonacci}. Below, their reverse trees $T\{-c_1,-c_2\}$ and~$T\{-c_1,-c_2,-c_3\}$ where the tipset for this last ternary tree is also the classical Rauzy fractal.
  }
  \label{treebonacci}
\end{figure}


\section{Concluding Remarks}

The methods to obtain and study one-parameter families of connected self-similar sets from stable complex trees are new. Our techniques also apply to arbitrary one-parameter families~$T_\mathcal{A}(z):=T\{c_1(z),c_2(z),\dots,c_n(z)\}$ like the ones pointed in the previous section which have stable set~$\mathcal{K}:=\{c\in\mathcal{R}:Q_\mathcal{A}(c)=\emptyset\}$. 
In these examples the resulting tipsets for parameters ~$c$ in the stable set~$\mathcal{K}$ are not that interesting from the point of view of analysis because all of these tipsets are disconnected and topologically homeomorphic to a Cantor set. On the other hand if~$T_\mathcal{A}(z)$ is tipset-connected for all~$z\in\mathcal{R}$ and~$\mathcal{K}\neq\emptyset$ then we know how to perform analysis over all the tipsets~$F_\mathcal{A}(z)$ for~$z\in\mathcal{K}$ once analysis has been developed for a single tipset~$F_\mathcal{A}(z)$ for~$z\in\mathcal{K}$. Therefore structurally stable complex trees and their associated region~$\mathcal{K}=\mathcal{R}\backslash\mathcal{M}$ are specially fit for the development of analysis on fractals and we believe that the study of these families is relevant to the current theory pioneered mainly by Jun Kigami~\cite{kigami2001analysis}.
\\

\noindent We have covered a small sample of examples of what is out there, more families of complex trees are being considered in~\cite{espigule2019ternary}~and~\cite{espigule2019Gternary}. We believe that the theoretical basis set in this work will contribute to the study of fractal dendrites~\cite{dendrites2018}~\cite{drozdov2018delta}~\cite{zeller2015branching}, topological spaces admitting a unique fractal structure~\cite{bandt1992topological}~\cite{donoven2016fractal}, and the geometry and dimension of fractals~\cite{bandt2006open}~\cite{falconer2004fractal}~\cite{vass2014geometry}~\cite{vass2018exact}. The notion of complex tree can be extended to hypercomplex spaces~\cite{frongillo2007symmetric}~\cite{bandt2010three}~\cite{espigule2013trees3D}. Future work will involve looking for quasisymmetric conjugacies between connected tipsets of complex trees and Julia sets. This is a direction that has been explored for certain types of self-similar sets~\cite{bandt1991self}~\cite{erouglu2010quasisymmetric}~\cite{kameyama2000julia}. 
 We also believe that the theory of complex dimensions and fractal strings developed by Michel L. Lapidus and Machiel van Frankenhuijsen~\cite{lapidus2012fractal} can be brought into the context of complex trees.
Finally, we think there are some real world applications awaiting for in technology, biology~\cite{benedetti2018engineered}~\cite{dale2014exploration}, physics, and other sciences.

\section*{Acknowledgments}

The author would like to thank Jofre Espigul\'e and J{\'o}zsef Vass for reviewing the preprint and sending valuable feedback to improve it. The author also thanks IMUB's Holomorphic Dynamics group, N\'uria Fagella, Xavier Jarque, Toni Garijo, Robert Florido, and all the other members, for all their support, dedication, feedback and valuable discussions during the author's research project on complex trees 2017-19. This project was partially supported by a research grant awarded by IMUB and the University of Barcelona. Early discussions with Stefano Silvestri and Vince Matsko about this approach using the notion of complex tree proved to be useful, the author would like to thank them for that.
 All figures and diagrams have been done with \textit{Mathematica}. The author thanks Wolfram Research and in particular Theodore Gray, Chris Carlson, Michael Trott, Vitaliy Kaurov, Todd Rowland, and Stephen Wolfram for showing interest on the author's early work, and for being a source of inspiration when it comes to computational experiments. Finally but not least, the author would like to express his gratitude to Susanne Kr\"omker, Pere Pascual, Gaspar Orriols, Warren Dicks, Tara Taylor, Hans Walser, Robert Fathauer, Tom Verhoeff, Michael Barnsley, Przemyslaw Prusinkiewicz, and many others for all their time and encouragement received from them during an early stage of the author's research 2012-2014 from which the present theory elaborates.
 
\bibliographystyle{unsrt}

\end{document}